\documentclass[leqno]{amsart}

\usepackage{amssymb}
\usepackage{amsthm}
\usepackage{array}
\usepackage{xparse}
\usepackage{enumitem}
\usepackage{tikz}
\usepackage{mathtools}
\usepackage[font=scriptsize]{caption}

\makeatletter
\def\black@#1{%
    \noalign{%
        \ifdim#1>\displaywidth
            \dimen@\prevdepth
            \nointerlineskip
            \vskip-\ht\strutbox@
            \vskip-\dp\strutbox@
            \vbox{\noindent\hbox to \displaywidth{\hbox to#1{\strut@\hfill}}}%
            \prevdepth\dimen@
        \fi
    }%
}
\makeatother

\newtheorem{thm}{Theorem}[section]

\newtheorem{lem}[thm]{Lemma}

\newtheorem{cor}[thm]{Corollary}

\theoremstyle{definition}

\theoremstyle{remark}

\newtheorem{rmk}[thm]{Remark}

\newcommand{\thmref}[1]{Theorem~\ref{#1}}
\newcommand{\corref}[1]{Corollary~\ref{#1}}
\newcommand{\secref}[1]{\S\ref{#1}}
\newcommand{\lemref}[1]{Lemma~\ref{#1}}

\newcommand*{\dis}{\displaystyle}
\newcommand*{\q}{\quad}
\newcommand*{\qq}{\qquad}

\newcommand*{\tx}[1]{\text{#1}}

\newcommand*{\ep}{\epsilon}

\newcommand*{\suchthat}{\, \middle| \,}

\newcommand*{\Ra}{\Rightarrow}

\newcommand*{\ftil}{\widetilde{f}}
\newcommand*{\fabs}{f_{abs}}
\newcommand*{\Pminus}{P_{-}}
\newcommand*{\Pminusbar}{\myoverline{-3}{0}{P}_{-}}

\newcommand*{\myoverline}[3]{\mkern -#1mu\overline{\mkern#1mu#3\mkern#2mu}\mkern -#2mu}	
\newcommand*{\vboldbar}{\myoverline{0}{0}{\vbold}}
\newcommand*{\Zbar}{\myoverline{-3}{0}{\Z}}

\newcommand*{\Fbar}{\myoverline{-3}{0}{F} }
\newcommand*{\Fzpbar}{\myoverline{-3}{-5}{F_\zp}}
\newcommand*{\Fzpepbar}{\myoverline{-3}{-5}{F_\zp^\ep}}
\newcommand*{\Bbar}{\myoverline{-3}{0}{B} }
\newcommand*{\Omegabar}{\myoverline{0}{0}{\Omega} }

\newcommand*{\half}{\frac{1}{2}}

\newcommand*{\Rsp}{\mathbb{R}}
\newcommand*{\Csp}{\mathbb{C}}

\newcommand*{\Lone}{L^1}
\newcommand*{\Ltwo}{L^2}

\newcommand*{\Linfty}{L^{\infty}}
\newcommand*{\Hhalf}{\dot{H}^\half}

\newcommand*{\al}{\alpha}
\newcommand*{\ap}{{\alpha'}}
\newcommand*{\apone}{\alpha_1'}
\newcommand*{\aptwo}{\alpha_2'}

\newcommand*{\be}{\beta}
\newcommand*{\bp}{{\beta'}}

\newcommand*{\xp}{{x'}}
\newcommand*{\yp}{{y'}}
\newcommand*{\zp}{{z'}}

\newcommand*{\diff}{\mathop{}\! d}

\newcommand*{\compose}[1]{\circ{#1}}

\newcommand*{\conv}{*}

\newcommand*{\Imag}{\tx{Im}}

\newcommand*{\grad}{\nabla}
\newcommand*{\Dt}{D_t}

\newcommand*{\pt}{\partial_t}

\newcommand*{\px}{\partial_x}
\newcommand*{\py}{\partial_y}
\newcommand*{\pz}{\partial_z}
\newcommand*{\pxp}{\partial_\xp}
\newcommand*{\pyp}{\partial_\yp}
\newcommand*{\pzp}{\partial_\zp}
\newcommand*{\pap}{\partial_\ap}

\newcommand*{\Ecal}{\mathcal{E}}

\newcommand*{\Ecalone}{\mathcal{E}_1}

\newcommand*{\vbold}{\mathbf{v}}

\newcommand*{\Pfrak}{\mathfrak{P}}
\newcommand*{\Pfrakep}{\Pfrak^\ep}

\newcommand*{\Aone}{A_1}

\newcommand*{\bvar}{b}
\newcommand*{\bap}{\bvar_\ap}
\newcommand*{\bvarap}{\bap}

\newcommand*{\h}{h}
\newcommand*{\hep}{\h^\ep}
\newcommand*{\hvart}{\h_t}
\newcommand*{\hal}{\h_\al}
\newcommand*{\halep}{\hal^\ep}

\newcommand*{\hinv}{\h^{-1}}

\newcommand*{\thvar}{\theta}

\newcommand*{\g}{g}
\newcommand*{\gep}{\g^\ep}

\newcommand*{\G}{G}
\newcommand*{\Gep}{\G^\ep}

\newcommand*{\F}{F}
\newcommand*{\Fep}{\F^\ep}

\newcommand*{\Fepbar}{\Fbar^\ep}
\newcommand*{\Fzp}{\F_{\zp}}
\newcommand*{\Fzpep}{\Fzp^\ep}

\newcommand*{\Ft}{\F_t}

\newcommand*{\Ftep}{\Ft^\ep}

\newcommand*{\Psiep}{\Psi^\ep}
\newcommand*{\Psizp}{\Psi_{\zp}}
\newcommand*{\Psizpep}{\Psi_{\zp}^\ep}
\newcommand*{\onePsizp}{\frac{1}{\Psizp}}
\newcommand*{\onePsizpep}{\frac{1}{\Psizpep}}

\newcommand*{\Psit}{\Psi_{t}}
\newcommand*{\Psitep}{\Psit^\ep}
\newcommand*{\Psiinv}{\Psi^{-1}}

\newcommand*{\z}{z}
\newcommand*{\zep}{\z^\ep}
\newcommand*{\zal}{\z_\al}
\newcommand*{\zalep}{\zal^\ep}

\newcommand*{\zt}{\z_t}
\newcommand*{\ztep}{\zt^\ep}

\newcommand*{\ztal}{\z_{t\al}}
\newcommand*{\ztalep}{\ztal^\ep}

\newcommand*{\ztt}{\z_{tt}}
\newcommand*{\zttep}{\ztt^\ep}

\newcommand*{\zttal}{\z_{tt\al}}

\newcommand*{\Z}{Z}
\newcommand*{\Zep}{\Z^\ep}
\newcommand*{\Zap}{\Z_{,\ap}}
\newcommand*{\oneZap}{\frac{1}{\Zap}}
\newcommand*{\Zapep}{\Zap^\ep}
\newcommand*{\oneZapep}{\frac{1}{\Zapep}}

\newcommand*{\Zapabs}{\abs{\Zap}}

\newcommand*{\Zt}{\Z_t}
\newcommand*{\Ztep}{\Zt^\ep}
\newcommand*{\Ztbar}{\Zbar_t}
\newcommand*{\Ztbarep}{\Zbar_t^\ep}
\newcommand*{\Ztap}{\Z_{t,\ap}}
\newcommand*{\Ztapep}{\Ztap^\ep}
\newcommand*{\Ztapbar}{\Zbar_{t,\ap}}
\newcommand*{\Ztapepbar}{\Ztapbar^\ep}

\newcommand*{\Ztt}{\Z_{tt}}
\newcommand*{\Zttep}{\Zep_{tt}}
\newcommand*{\Zttbar}{\Zbar_{tt}}
\newcommand*{\Zttbarep}{\Zttbar^\ep}
\newcommand*{\Zttap}{\Z_{tt,\ap}}

\newcommand*{\nobrac}[1]{ #1 }

\DeclarePairedDelimiter{\oldbrac}{\lparen}{\rparen}			
\NewDocumentCommand{\brac}{ s o m }{						
	\IfBooleanT{#1}{
  		\IfValueT{#2}{\oldbrac[#2]{#3}}
		\IfValueF{#2}{\oldbrac{#3}} 
	}
	\IfBooleanF{#1}{
  		\IfValueT{#2}{\PackageError{mypackage}{Incorrect use of brac. Insert star}{}}
		\IfValueF{#2}{\oldbrac*{#3}} 
	}		
}

\DeclarePairedDelimiter\oldcbrac{\lbrace}{\rbrace}				
\NewDocumentCommand{\cbrac}{ s o m }{					
	\IfBooleanT{#1}{
  		\IfValueT{#2}{\oldcbrac[#2]{#3}}
		\IfValueF{#2}{\oldcbrac{#3}} 
	}
	\IfBooleanF{#1}{
  		\IfValueT{#2}{\PackageError{mypackage}{Incorrect use of cbrac. Insert star}{}}
		\IfValueF{#2}{\oldcbrac*{#3}} 
	}		
}

\DeclarePairedDelimiter\oldsqbrac{\lbrack}{\rbrack}				
\NewDocumentCommand{\sqbrac}{ s o m }{					
	\IfBooleanT{#1}{
  		\IfValueT{#2}{\oldsqbrac[#2]{#3}}
		\IfValueF{#2}{\oldsqbrac{#3}} 
	}
	\IfBooleanF{#1}{
  		\IfValueT{#2}{\PackageError{mypackage}{Incorrect use of sqbrac. Insert star}{}}
		\IfValueF{#2}{\oldsqbrac*{#3}} 
	}		
}

\DeclarePairedDelimiter{\oldabs}{\lvert}{\rvert}
\NewDocumentCommand{\abs}{ s o m }{						
	\IfBooleanT{#1}{
  		\IfValueT{#2}{\oldabs[#2]{#3}}
		\IfValueF{#2}{\oldabs{#3}} 
	}
	\IfBooleanF{#1}{
  		\IfValueT{#2}{\PackageError{mypackage}{Incorrect use of abs. Insert star}{}}
		\IfValueF{#2}{\oldabs*{#3}} 
	}		
}

\DeclarePairedDelimiterX{\oldnorm}[1]{\lVert}{\rVert}{#1}
\NewDocumentCommand{\norm}{ s o o m }{					
	\IfValueT{#2} {
		\IfBooleanT{#1}{
  			\IfValueT{#3}{\oldnorm[#2]{#4}_{#3}}
			\IfValueF{#3}{\oldnorm{#4}_{#2}} 
		}
		\IfBooleanF{#1}{
  			\IfValueT{#3}{\PackageError{mypackage}{Incorrect use of norm. Insert star}{}}
			\IfValueF{#3}{\oldnorm*{#4}_{#2}} 
		}
	}
	\IfValueF{#2} {
		\IfBooleanT{#1}{\oldnorm{#4}}	
		\IfBooleanF{#1}{\oldnorm*{#4}}		
	}	
}

\calclayout

\begin{document}

\title[water waves]{Rigidity of singularities of 2D gravity water waves}
\author{Siddhant Agrawal}
\address{Department of Mathematics, University of Michigan, Ann Arbor, MI 48104}
\email{sidagr@umich.edu}

\begin{abstract}
We consider the Cauchy problem for the 2D gravity water wave equation. Recently Wu \cite{Wu15, Wu18}  proved the local well-posedness of the equation in a regime which allows interfaces with angled crests as initial data. In this work we study properties of these singular solutions and prove that the singularities of these solutions are "rigid". More precisely we prove that an initial interface with angled crests remains angled crested, the Euler equation holds point-wise even on the boundary, the particle at the tip stays at the tip, the acceleration at the tip is the one due to gravity and the angle of the crest does not change  nor does it tilt. We also show that the existence result of Wu \cite{Wu15} applies not only to interfaces with angled crests, but also allows certain types of cusps. 
\end{abstract}

\maketitle


\section{Introduction}

We are concerned with the motion of a fluid in dimension two with a free boundary. In this work we will identify 2D vectors with complex numbers. The fluid region $\Omega(t) \subset \Csp$ is assumed to be a simply connected open set with boundary $\Sigma(t)$, which is assumed to be a Jordan curve. The bottom is assumed to be at infinite depth and the interface $\Sigma(t)$ is assumed to tend to the real line at infinity. We do not assume that the interface $\Sigma(t)$ is a graph. The fluid region $\Omega(t)$ and the air is separated by the interface $\Sigma(t)$, with the fluid being below the air region.  The air and the fluid are assumed to have constant densities of 0 and 1 respectively. The fluid is also assumed to be inviscid, incompressible and irrotational.  The gravitational field is assumed to be a constant vector $-i$ pointing in the downward direction and surface tension is neglected. The motion of the fluid is then governed by the Euler equation
\begin{align}\label{eq:Euler}
\begin{cases}
\ \mathbf{v_t + (v.\nabla)v} = -i -\nabla P  \qq\text{ on } \Omega (t) 	\\
\  \tx{div } \mathbf{ v} = 0, \quad \tx{curl } \mathbf{ v }=0 \qq\enspace\text{ on } \Omega(t) \\
\   P = 0 \qq\qq\qq\qq\qq\text{ on } \Sigma (t) \\
\   (1,\mathbf{v}) \tx{ is tangent to the free surface } (t, \Sigma(t))
\end{cases}
\end{align}
along with the decay conditions $\vbold \to 0 $, $\grad P \to -i$ as $\abs{(x,y)} \to \infty$. The last boundary condition i.e. $(1,\mathbf{v}) \tx{ is tangent to the free surface } (t, \Sigma(t))$, is equivalent to the condition that particles on the boundary stay on the boundary.

The study of special solutions of water waves has a long history. More than a century ago, Stokes \cite{St80} formally constructed traveling wave solutions with sharp crests of angle $2\pi/3$. The existence of these singular waves was proven by Toland \cite{To78} and Amick, Fraenkel, and Toland \cite{AmFrTo82} proved that these singular solutions do indeed have sharp crests of angle $2\pi/3$. See also the works in  \cite{PlTo04, Co12}. In the case of zero gravity, Wu \cite{Wu12} proved the existence of self-similar solutions which have sharp crests of angle $\nu\pi$ where $0<\nu<\half$.

The earliest results on local well-posedness in Sobolev spaces for the Cauchy problem are for small data in 2D and were obtained by Nalimov \cite{Na74}, Yoshihara \cite{Yo82,Yo83} and Craig \cite{Cr85}. Wu \cite{Wu97,Wu99} proved the local well-posedness for arbitrary data in Sobolev spaces, in the infinite depth case. Later this result was extended to the case of bottom with finite depth, non-zero surface tension, non-zero vorticity and in lowering the Sobolev regularity in  \cite{ChLi00, La05, Li05, AmMa05, CoSh07, ShZe08, ZhZh08, MiZh09, AmMa09, ABZ14a, HuIfTa16, ABZ14b, Poy16}. 

Kinsey and Wu \cite{KiWu18} proved an a priori estimate for an energy $\Ecalone(t)$ which allows both smooth interfaces and interfaces with angled crests. Building upon this work, Wu \cite{Wu15} obtained local existence of solutions for initial data with $\Ecalone(0) < \infty$. Recently Wu \cite{Wu18} proved the uniqueness of these solutions.

In this work we study the properties of the solutions with singular initial data constructed in \cite{Wu15}. In \cite{KiWu18} it is observed that the angle of the angled crests should not change with time, and a heuristic argument is given to support this claim. In this work, we rigorously prove this claim and further study the nature of the solution around the singularities. If the solution constructed in \cite{Wu15} exists in the time interval $[0,T]$, we show that during this time interval the following is true:
\begin{enumerate}
\item An interface with angled crests at $t=0$ remains angled crested for $t>0$.
\item The particle at the tip stays at the tip.
\item The quantities $ \grad \vbold(\cdot,t)$ and $\grad P(\cdot,t)$ extend continuously to $\Omegabar(t)$ and the Euler equation holds on $\Omegabar(t)$. \footnote{In \cite{Wu15} it was shown that $\vbold(\cdot,t)$ and $P(\cdot,t)$ extends continuously to $\Omegabar(t)$ and the Euler equation holds in the strong sense in the interior of $\Omega(t)$.}
\item  $\grad \vbold(\cdot,t)$ and $\grad P(\cdot,t)$ vanish at the singularities. Hence the acceleration at the tip is $-i$.
\item The angle of the crest does not change, nor does it tilt i.e. the left and right unit tangent vectors at the crest do not change with time.
\end{enumerate}
See the section on main results \secref{sec:results} for a more precise formulation of the results. We also show in \secref{sec:exampleE} that the existence result in \cite{Wu15} applies not only to interfaces with angled crests, but also allows certain types of cusps.

The paper is organized as follows: In  \secref{sec:sing} we introduce the notation and explain the existence result of Wu \cite{Wu15}. We also give a heuristic explanation for our results. In \secref{sec:results} we state our main results. In \secref{sec:proof} we prove our results. In \secref{sec:exampleE} we give some examples of initial data for which our results apply.

\bigskip

\noindent \textbf{Acknowledgment}: This work is part of my Ph.D. thesis and I am very grateful to my advisor Prof. Sijue Wu for helpful discussions. The author was supported in part by NSF grants DMS-1101434, DMS-1361791.

\section{Notation and preliminaries}\label{sec:sing}

In this section we first introduce the notation, then briefly recall the existence result in \cite{Wu15} and finally give a heuristic explanation for our results. We will try to be as consistent as possible with the notation used in \cite{Wu15}. The definitions and notation introduced here are for smooth solutions to the Euler equation. For singular solutions, we will clarify which definitions make sense. The Fourier transform is defined as
\[
\hat{f}(\xi) = \frac{1}{\sqrt{2\pi}}\int e^{-ix\xi}f(x) \diff x
\]
The Sobolev spaces $H^s(\Rsp)$ for $s\geq 0$  is the space of functions with  $\norm[H^s]{f} = \norm*[2]{(1+\abs{\xi}^2)^{\frac{s}{2}}\hat{f}(\xi)} < \infty$. The homogenous Sobolev space $\Hhalf(\Rsp)$ is the space of functions modulo constants with  $\norm[\Hhalf]{f} = \norm*[2]{\abs{\xi}^\half \hat{f}(\xi)} < \infty$. Define the Poisson kernel by
\begin{align*}
K_y(x) = \frac{-y}{\pi(x^2+y^2)} \qq   y<0
\end{align*}

From now on compositions of functions will always be in the spatial variables. We write $f = f(\cdot,t), g = g(\cdot,t), f \compose g(\cdot,t) :=  f(g(\cdot,t),t)$. We will denote the spacial coordinates in $\Omega(t) $ with $z = x+iy$, whereas $\zp = \xp + i\yp$ will denote the coordinates in the lower half plane $\Pminus = \cbrac{(x,y) \in \Rsp^2 \suchthat y<0}$. As we will frequently work with holomorphic functions, we will use the holomorphic derivatives $\pz = \half(\px-i\py)$ and $\pzp = \half(\pxp-i\pyp)$. Let the interface $\Sigma(t) : \z = \z(\al, t) \in \Csp$ be given by a Lagrangian parametrization with parameter $\al$ satisfying $\z_{\al}(\al,t) \neq 0 $ for all  $\al \in \Rsp$. Hence $\zt(\al,t) = \mathbf{v}(\z(\al,t),t)$ is the velocity of the fluid on the interface and $\ztt(\al,t) = (\mathbf{v_t + (v.\nabla)v})(\z(\al,t),t)$ is the acceleration.

Let $\Psi(.,t): \Pminus \to  \Omega(t)$ be Riemann maps satisfying $\lim_{\z\to \infty} \Psi_\z(\z,t) =1$  and $\lim_{\z\to \infty} \Psi_t(\z,t) =0$.  With this, the only ambiguity left in the definition of $\Psi$ is that of the choice of translation of the Riemann map at $t=0$, which does not play any role in the analysis. Let $\Phi(\cdot,t):\Omega(t) \to \Pminus $ be the inverse of the map $\Psi(\cdot,t)$ and define
\begin{align}\label{def:h}
\h(\al,t) = \Phi(\z(\al,t),t)
\end{align}
hence $\h(\cdot,t):\Rsp \to \Rsp$ is a homeomorphism. The map $h(\cdot,t)$ connects the Lagrangian and conformal parameterizations. As we use both Lagrangian and conformal parameterizations, we will denote the Lagrangian parameter by $\al$ and the conformal parameter by $\ap$. Let $\hinv(\cdot,t)$ be the inverse of $h(\cdot,t)$ i.e.
\[
\h(\hinv(\ap,t),t) = \ap
\]
 From now on, we will fix our Lagrangian parametrization at $t=0$ by imposing
\begin{align*}
h(\al,0)= \al \quad  \tx { for all } \al \in \Rsp
\end{align*}
Hence the Lagrangian parametrization is the same as the conformal parametrization at $t=0$. Define the variables
\[
\begin{array}{l l l}
 \Z(\ap,t) = \z\compose \hinv (\ap,t)  & \Zap(\ap,t) = \pap \Z(\ap,t) &  \quad  \tx{ Hence } \quad  \brac*[]{\dfrac{\zal}{\hal}} \compose \hinv = \Zap \\
  \Zt(\ap,t) = \zt\compose \hinv (\ap,t)  & \Ztap(\ap,t) = \pap \Zt(\ap,t) &  \quad   \tx{ Hence } \quad  \brac*[]{\dfrac{\ztal}{\hal}} \compose \hinv = \Ztap \\
 \Ztt(\ap,t) = \ztt\compose \hinv (\ap,t)  & \Zttap(\ap,t) = \pap \Ztt(\ap,t) &  \quad   \tx{ Hence } \quad  \brac*[]{\dfrac{\zttal}{\hal}} \compose \hinv = \Zttap \\
\end{array}
\]
Hence $\Z(\ap,t), \Zt(\ap,t)$ and $\Ztt(\ap,t)$  are the parametrizations of the boundary, the velocity and the acceleration in conformal coordinates and in particular $\Z(\cdot,t)$ is the boundary value of the Riemann map $\Psi(\cdot,t)$. Note that as $\Z(\ap,t) = \z(\hinv(\ap,t),t)$ we see that  $\pt \Z \neq \Zt$. Similarly $\pt \Zt \neq \Ztt$.  In conformal coordinates, the substitute for the time derivative is the material derivative $\Dt = \pt + \bvar\pap$ where $\bvar = \hvart\compose \hinv$. Then we have $\Dt \Z = \Zt$ and $\Dt \Zt = \Ztt $. If we take complex conjugate of the Euler equation \eqref{eq:Euler} we get
\begin{align}\label{eq:Eulerconj}
\vboldbar_t + (\vbold\cdot\grad)\vboldbar =  -(\px - i\py)P + i \quad \tx{ on } \Omega(t) 
\end{align}
 Define $\F: \Pminusbar \to \Csp $ and $\Pfrak: \Pminusbar \to \Rsp$ as
\begin{align*}
\F = \vboldbar \compose \Psi \qquad \Pfrak = P \compose \Psi
\end{align*}
and observe that $\F$ is a holomorphic function on $\Pminus$. Also note that its boundary value is given by $ \Ztbar(\ap,t) = \F(\ap,t)$ for all $\ap\in \Rsp$.  As $(\vbold \cdot\grad)\vboldbar = \vbold\vboldbar_z $, we see that \eqref{eq:Eulerconj} can be written as
\begin{align*}
(\F\compose \Psiinv)_t + \Fbar\compose\Psiinv (\F\compose\Psiinv)_z = -(\px - i\py)(\Pfrak\compose \Psiinv) + i 
\end{align*} 
Hence by chain rule and changing coordinates we can rewrite the Euler equation as 
\begin{align}\label{eq:EulerRiem}
\begin{cases}
\ \Ft - \Psit \frac{\Fzp}{\Psizp} + \Fbar \frac{\Fzp}{\Psizp} = -\frac{1}{\Psizp}(\pxp - i\pyp)\Pfrak + i \quad\! \tx{ on } \Pminus \\
\ \F(\cdot,t) \tx{ is holomorphic} \qq\qq\qq\qq\qq\q \tx{ on } \Pminus \\
\ \Pfrak = 0 \qq\qq\qq\qq\qq\qq\qq\qq\q\enspace\, \tx{ on } \partial \Pminus \\
\ \tx{Trace of } \onePsizp(\Fbar - \Psit) \tx{ is real valued} \qq\qq\q  \tx{ on } \partial \Pminus
\end{cases}
\end{align}
along with the condition that $\Psi(\cdot,t)$ is conformal and the decay conditions $\F \to 0$, $(\pxp - i\pyp) \Pfrak \to i$, $\Psizp \to 1$ and $ \Psi_t \to 0$ as $\zp \to \infty$.

Note from \eqref{eq:EulerRiem} that the material derivative on $\Pminus$ is given by the operator $\pt + \onePsizp(\Fbar - \Psit)\pzp$. Hence the condition that the particles on the boundary stay on the boundary is equivalent to saying that the trace of $\onePsizp(\Fbar - \Psit)$ is real valued.

It should be noted that the process of obtaining \eqref{eq:EulerRiem} from \eqref{eq:Euler} is reversible so long as the interface $\Sigma(t) = \cbrac{\Z(\ap,t) \suchthat \ap\in \Rsp}$ is Jordan curve. Note that the boundary value of left hand side of the first equation \eqref{eq:EulerRiem} is the conjugate of the acceleration in conformal coordinates and hence
\begin{align*}
\Zttbar(\ap,t) = \brac{\Ft - \Psit \frac{\Fzp}{\Psizp} + \Fbar \frac{\Fzp}{\Psizp}}(\ap,t) \qq \tx{ on } \partial\Pminus
\end{align*}
The boundary value of the conjugate of the gradient of the pressure is given by the following formula from \cite{Wu15}
\begin{align}\label{eq:gradPbdry}
i \frac{\Aone}{\Zap}(\ap,t) = \brac{\frac{1}{\Psizp}(\pxp - i\pyp)\Pfrak}(\ap,t) \qq \tx{ on } \partial\Pminus
\end{align}
where $\Aone(\cdot,t):\Rsp \to \Rsp$ is defined as
\begin{align*}
\Aone(\ap,t) = 1 + \frac{1}{2\pi}\int \frac{\abs{\Zt(\ap,t) - \Zt(\bp,t)}^2}{(\ap-\bp)^2} \diff\bp 
\end{align*}
Hence equation \eqref{eq:EulerRiem} on the boundary is
\begin{align}\label{form:Zttbar}
\Zttbar -i = -i \frac{\Aone}{\Zap}
\end{align}

\medskip
\noindent \textbf{Existence of solutions with singular initial data}
\medskip

From now on we will use the following convention: If $H:\Pminus \to \Csp$ is continuous function which extends continuously to $\Pminusbar$, then we say that $H$ is continuous on $\Pminusbar$ with its boundary value given  by the continuous extension. This convention becomes especially important when there are two or more functions involved. For e.g. if $L,M: \Pminus \to \Csp$ are two continuous functions on $\Pminus$ such that $(LM)(\zp)$ extends continuously to $\Pminusbar$, then we say that $LM$ is a continuous function on $\Pminusbar$ with its boundary values denoted by $(LM)(\ap)$. Note that in such a situation, the boundary values of $L$ and $M$ may not be defined and hence it makes no sense apriori to talk about the equality $(LM)(\ap) = L(\ap)M(\ap)$. 

We now describe the existence result of Wu namely Theorem 3.4 in \cite{Wu15} which applies to both smooth and singular initial data. This existence result was reproved in \cite{Wu18} with a few minor modifications. The solutions constructed in \cite{Wu15,Wu18} solve equation \eqref{eq:EulerRiem} in $\Pminus$ in the strong sense, and give rise to a physical solution to the Euler equation \eqref{eq:Euler} if the conformal maps $\Psi(\cdot,t)$ are invertible. These solutions are constructed via an approximation by smooth solutions in the following energy class 
\begin{align*}
\Ecalone(t) & =   \sup_{\yp<0}\norm[\Ltwo(\Rsp,\diff \xp)]{\Fzp(t)}^2 + \sup_{\yp<0}\norm[\Hhalf(\Rsp,\diff \xp)]{\onePsizp\Fzp(t)}^2 + \sup_{\yp<0}\norm[\Ltwo(\Rsp, \diff \xp)]{\frac{1}{\Psizp}\pzp\brac{\frac{1}{\Psizp}\Fzp}(t)}^2 \\
& \quad  + \sup_{\yp<0}\norm[\Hhalf(\Rsp, \diff \xp)]{\frac{1}{\Psizp^2}\pzp\brac{\frac{1}{\Psizp}\Fzp}(t)}^2 + \sup_{\yp<0}\norm[\Linfty(\Rsp,\diff \xp)]{\frac{1}{\Psizp}(t)}^2  \\
& \quad  + \sup_{\yp<0}\norm[\Ltwo(\Rsp, \diff \xp)]{\partial_\zp \brac{\frac{1}{\Psizp}}(t)}^2 + \sup_{\yp<0}\norm[\Ltwo(\Rsp, \diff \xp)]{\frac{1}{\Psizp}\pzp\brac{\frac{1}{\Psizp}\pzp \brac{\frac{1}{\Psizp}}}(t)}^2
\end{align*}
Observe that if $\F$ and $\onePsizp$ extend smoothly to the boundary, then this energy is equal to the energy
\begin{align*}
\Ecal(t) & =  \norm[2]{\Ztapbar(t)}^2 +  \norm[\Hhalf]{\frac{\Ztapbar}{\Zap}(t)}^2  + \norm[2]{\oneZap\pap\brac{\oneZap\Ztapbar}(t)}^2 + \norm[\Hhalf]{\frac{1}{\Zap^2}\pap\brac{\oneZap\Ztapbar}(t)}^2\\
& \quad  +  \norm[\infty]{\frac{1}{\Zap}(t)}^2 + \norm[2]{\pap\frac{1}{\Zap}(t)}^2  + \norm[2]{\oneZap\pap\brac{\oneZap\pap\frac{1}{\Zap}}(t)}^2 \\
\end{align*}
where all these quantities are defined on the boundary $\partial\Pminus$. It is important to note that the energy $\Ecalone(t)$ allows both smooth and singular domains, including domains with angled crests and cusps. See \secref{sec:exampleE} for more details. In \cite{KiWu18} an apriori energy estimate is proved for the energy $\Ecal(t)$ for smooth enough solutions. 

The initial data $(\Psi,\F,\Pfrak)(0)$ is chosen so that $\Psi(\cdot, 0), \F(\cdot,0) :\Pminus \to \Csp$ are holomorphic with $\Psi(\cdot,0)$ being conformal, the range of $\Psi(\cdot,0)$ namely $\Omega(0) = \cbrac{\Psi(\zp,0) \suchthat  \zp \in \Pminus}$ is a domain with the boundary $\Sigma(0)$ being a Jordan curve, and $\lim_{\zp \to \infty} \Psizp(\zp,0) = 1$. $\Pfrak(\cdot,0):\Pminus \to \Rsp$ is chosen so that it is the unique solution to
\begin{align*}
\Delta \Pfrak = -2\abs{\Fzp}^2 \quad \tx{ on } \Pminus \qq \tx{ and } \Pfrak = 0 \quad \tx{ on } \partial \Pminus 
\end{align*}
with the condition $(\pxp + i\pyp)\Pfrak \to -i$ as ${\zp \to \infty}$.  It is assumed that $\Ecalone(0)<\infty$ along with 
\begin{align*}
c_0 = \sup_{\yp <0}\norm[\Ltwo(\Rsp,\diff \xp)]{\F(\xp + i\yp,0)}  + \sup_{\yp <0}\norm[\Ltwo(\Rsp,\diff \xp)]{\frac{1}{\Psizp(\xp+i\yp,0)} - 1} < \infty
\end{align*}
For such initial data, Wu proved the following existence result.

\begin{thm}[\cite{Wu15,Wu18}]\label{thm:Wuexistence}
Let the initial data $(\F,\Psi,\Pfrak)(0)$ be described as above with $\Ecalone(0)<\infty$. Then there exists a time $T_0 >0$ depending only on $\Ecalone(0)$ such that on $[0,T_0]$ the initial value problem of the gravity water wave equation \eqref{eq:EulerRiem} has a solution $(\F,\Psi,\Pfrak)(t)$ with the following properties:
\begin{enumerate}
\item $\Psi(\cdot,t)$ is conformal on $\Pminus$ for each fixed $t\in[0,T_0]$, $\Psi$ is continuous differentiable on $\Pminus\times[0,T_0]$ and $\Psi$, $ \frac{1}{\Psizp}$ and $\frac{\Psit}{\Psizp}$ are continuous on $\Pminusbar\times [0,T_0]$

\item $\F(\cdot,t)$ is holomorphic on $\Pminus$ for each fixed $t\in[0,T_0]$, $\F$ is continuous on $\Pminusbar\times [0,T_0]$ and $\F$ is continuous differentiable on $\Pminus\times [0,T_0]$

\item $\Pfrak$ is continuous on $\Pminusbar\times [0,T_0]$ and $\Pfrak$ is continuous differentiable on $\Pminus\times [0,T_0]$

\item For all $t \in [0,T_0]$ we have $\Ecalone(t)<\infty$ and 
\begin{align}\label{lowerreg}
\sup_{\yp <0}\norm[\Ltwo(\Rsp,\diff \xp)]{\F(\xp + i\yp,t)}  + \sup_{\yp <0}\norm[\Ltwo(\Rsp,\diff \xp)]{\frac{1}{\Psizp(\xp+i\yp,t)} - 1} < \infty
\end{align}
\end{enumerate}
The solution gives rise to a solution $(\vboldbar, P) = (\F\compose\Psi^{-1}, \Pfrak\compose\Psi^{-1} )$ of the water wave equation $\eqref{eq:Euler}$ so long as $\Sigma(t) = \cbrac{Z = \Psi(\ap,t) \suchthat \ap\in \Rsp}$ is a Jordan curve. Moreover if the initial interface is chord-arc, that is $\Zap(\cdot,0) \in \Lone_{loc}(\Rsp)$ and there is constant $0<\delta<1$, such that
\begin{align*}
\delta\int_\ap^\bp \abs{\Zap(\gamma,0)}\diff\gamma \leq \abs{\Z(\ap,0) - \Z(\bp,0)} \leq \int_\ap^\bp \abs{\Zap(\gamma,0)}\diff\gamma \qq \forall -\infty<\ap<\bp<\infty
\end{align*}
then there is $T_0>0$, $T_1>0$ with $T_0,T_1$ depending only on $\Ecalone(0)$, such that on $[0,\min\cbrac*[\big]{T_0,\frac{\delta}{T_1}}]$, the initial value problem of the water wave equation \eqref{eq:Euler} has a solution satisfying $\Ecalone(t)<\infty$ and \eqref{lowerreg}, and the interface $Z = Z(\cdot,t)$ is chord arc.

\end{thm}
\begin{rmk}
The solutions constructed in \thmref{thm:Wuexistence} were shown to be unique in a suitable class of solutions which can be approximated by smooth solutions. See \cite{Wu18} for more details.
\end{rmk}

\thmref{thm:Wuexistence} is proved by first mollifying the initial data, then solving the initial value problem to the boundary equation \eqref{form:Zttbar} and showing that the smooth solution exists in a the time interval $[0,T_0]$, where $T_0$ is independent of the mollification parameter $\ep$. The fact that $T_0$ is independent of $\ep$ is proved using the apriori estimate proved in \cite{KiWu18} and by a blow up criterion. These solutions give rise to solutions of the equation \eqref{eq:EulerRiem} in the interval $[0,T_0]$, and a solution for data satisfying $\Ecalone(0)<\infty$ is constructed by taking a limit as $\ep\to0$.  As we will need some elements of the proof, we describe some of the notation and statements proved in the proof of this result.

Let the initial data be $(\Psi,\F,\Pfrak)(0)$ and let $0<\ep\leq 1$. Define
\begin{align*}
\Z^\ep (\ap,0) = \Psi(\ap-\ep i,0), \quad \Ztbar^\ep (\ap,0) = \F(\ap-\ep i,0),\quad  \h^\ep(\al,0) = \al \\
\F^\ep(\zp,0) = \F(\zp-\ep i,0), \quad\Psi^\ep(\zp,0) = \Psi(\zp - \ep i,0) \qquad
\end{align*}
Similarly define $\bvar^\ep = h_t^\ep\compose(\h^\ep)^{-1}$. Also define
\begin{align*}
\Aone^\ep = 1 + \frac{1}{2\pi}\int \frac{\abs{\Ztep(\ap,t) - \Ztep(\bp,t)}^2}{(\ap-\bp)^2} \diff\bp 
\end{align*}

 Then as part of the existence result of the smooth solution $(\Z^\ep,\Ztbar^\ep)(t)$ in $[0,T_0]$ of the equation \eqref{form:Zttbar}, it is shown that (see \cite{KiWu18,Wu15}) for all $0<\ep\leq 1$ and $t\in[0,T_0]$ we have
\begin{align}\label{ineq:quantE}
\begin{split}
 \norm[2]{\Ztapep} + \norm[2]{\pap\frac{1}{\Zapep}} +  \norm[\infty]{\Aone^\ep} + \norm[2]{\oneZapep\pap\Aone^\ep}  &  \leq C(\Ecalone(0)) \\
 \norm[\infty]{\bvarap^\ep} + \norm[\infty]{\frac{\h_{t\al}^\ep}{\hal^\ep}}  + \norm[\infty]{\frac{\ztal^\ep}{\zal^\ep}} + \norm[2]{\pap\brac*[\bigg]{ \frac{1}{(\Zap^\ep)^2}\Ztap^\ep}} & \leq C(\Ecalone(0))
\end{split}
\end{align}
and also
\begin{align}\label{ineq:quantE2}
 \norm[\infty]{\Ztep} + \norm[\infty]{\frac{1}{\Zapep}} \leq C(c_0,\Ecalone(0))
\end{align}

Using these it can be easily seen that there exists constants $0<c_1,c_2 < \infty$ depending only on $\Ecalone(0)$ such that for all $(\al,t) \in \Rsp\times[0,T_0]$ and for all $0<\ep\leq 1$ we have
\begin{align*}
c_1 \leq \abs{\hal^\ep(\al,t)} \leq c_2   \\
\tx{and } c_1\abs{\zal^\ep}(\al,0) \leq \abs{\zal^\ep}(\al,t) \leq c_2\abs{\zal^\ep}(\al,0)  \\
\tx{and }  \abs{\ztal^\ep}(\al,t) \leq c_2\abs{\zal^\ep}(\al,0) 
\end{align*}
If $U\subset \Rsp^n$ then we will use the notation $f_n \Rightarrow f$ on $U$ to mean uniform convergence on compact subsets of $U$.  In the proof of \thmref{thm:Wuexistence} a subsequence is taken $\ep_j \to 0$ and for convenience it is replaced by $\ep$. In the proof of our main results in the next section, we will also use this notation and we will freely take a subsubsequence of the subsequence used in \cite{Wu15} as it does not affect the result. As part of the proof of \thmref{thm:Wuexistence}, the following statements are proved as $\ep = \ep_j \to 0$

\begin{enumerate}[label=(\alph*)]
\item There exists a continuous function $h:\Rsp\times[0,T_0] \to \Rsp$ so that $h(\cdot,t):\Rsp \to \Rsp$ is a homeomorphism and
\begin{align*}
h^\ep \Rightarrow h \quad \tx{ and } \quad (h^\ep)^{-1} \Rightarrow h^{-1} \qq \tx{ on } \Rsp\times[0,T_0]
\end{align*}
Moreover there exists constants $c_1, c_2 > 0$ depending only on $c_0$ and $\Ecalone(0) $ so that
\begin{align*}
0< c_1 \leq \frac{h(\al,t) - h(\be,t)}{\al - \be} \leq c_2 < \infty \qq \tx{ for all }  \al,\be \in \Rsp \tx{ with }\al \neq \be \tx{ and } t \in [0,T_0]
\end{align*}

\item There exists a continuous function $z:\Rsp\times[0,T_0] \to \Csp$ such that $z$ is twice continuously differentiable with respect to $t$, with $\zt$ and $\ztt$ being continuous and bounded functions on $\Rsp\times[0,T_0]$ satisfying
\begin{align*}
\zep \Ra \z \qq \ztep \Ra \zt \qq \zttep \Ra\ztt \qq \tx{ on } \Rsp\times[0,T_0]
\end{align*}

\item If $\Z(\ap,t) = \z(\hinv(\ap,t),t) $, $\Zt(\ap,t) = \zt(\hinv(\ap,t),t) $ and $\Ztt(\ap,t) = \ztt(\hinv(\ap,t),t) $, then observe that $\Z,\Zt$ and $\Ztt$ are continuous functions on $\Rsp\times[0,T_0]$ with $\Zt$ and $\Ztt$ being bounded as well. We also have
\begin{align*}
\Zep \Ra \Z \qq \Ztep \Ra \Zt \qq \Zttep \Ra \Ztt \qq \tx{ on } \Rsp\times[0,T_0]
\end{align*}

\item There exists a continuous function $\Psi:\Pminusbar \times [0,T_0] \to \Csp$ such that $\Psi(\cdot,t)$ is conformal on $\Pminus$ and $\onePsizp$ extends continuously to $\Pminusbar$. Its boundary value is given by $\Z(\ap,t) = \Psi(\ap,t)$ and we also have
\begin{align*}
\Psiep \Ra \Psi \qq \onePsizpep \Ra \onePsizp \qq \tx{on } \Pminusbar \times[0,T_0]\\
\Psitep \Ra \Psit \qq  \Psizpep \Ra \Psizp  \qq \tx{on } \Pminus \times[0,T_0]
\end{align*}
The proof also shows that $\frac{\Psit}{\Psizp}$ extends continuously to $\Pminusbar$ and we have
\begin{align*}
\frac{\Psitep}{\Psizpep} \Ra \frac{\Psit}{\Psizp} \qq \tx{ on } \Pminusbar\times[0,T_0]
\end{align*}

\item There exists a continuous and bounded function $\F:\Pminusbar \times [0,T_0] \to \Csp$ such that $\F(\cdot,t)$ is holomorphic on $\Pminus$. Its boundary value is given by $\Ztbar(\ap,t) = \F(\ap,t)$ and we also have
\begin{align*}
\Fep \Ra \F \qq \tx{ on } \Pminusbar \times[0,T_0] \\
\Fzpep \Ra \Fzp \qq \tx{ on } \Pminus \times[0,T_0]
\end{align*}
Also $\F$ is continuously differentiable with respect to $t$ with $\Ftep \Ra \Ft$ on $\Pminus\times[0,T_0]$. 

\item The pressure $\Pfrakep$ for the smooth solutions is given by 
\begin{align}\label{form:Pfrakep}
\Pfrakep(\zp,t) = -\half\abs{\Fep(\zp,t)}^2 - \yp +  \half K_\yp\conv (\abs{\Ztbarep}^2)(\xp,t) 
\end{align}
with $\Pfrakep = 0 $ on $\partial \Pminus$. Define the function $\Pfrak:\Pminus\times[0,T_0] \to \Rsp$ by
\begin{align}\label{form:Pfrak}
\Pfrak(\zp,t) = -\half\abs{\F(\zp,t)}^2 - \yp +  \half K_\yp\conv (\abs{\Ztbar}^2)(\xp,t) 
\end{align}
Then $\Pfrak$ extends to  a continuous function on $\Pminusbar\times[0,T_0]$  with $\Pfrak \in C([0,T_0],C^\infty(\Pminus))$ and $\Pfrak = 0 $ on $\partial \Pminus$. We also have
\begin{align*}
\Pfrakep \Ra \Pfrak \qq \tx{ on } \Pminusbar \times[0,T_0] \\
(\pxp - i\pyp)\Pfrakep \Ra (\pxp - i\pyp)\Pfrak \qq \tx{ on } \Pminus \times[0,T_0]
\end{align*}

\end{enumerate}

\medskip
\noindent \textbf{Heuristics}
\medskip

We now explain the heuristics about the evolution of the angle as stated in \cite{KiWu18} and also give a heuristic explanation about the nature of $\grad P$ near the singularities. Let $\nu>0$ and assume that there is angled crest of angle $\nu\pi$ at $\ap = 0$ in conformal coordinates. Hence as $\ap \to 0$, we have
\begin{align*}
\Z (\ap) \sim (\ap)^\nu \qq \Zap (\ap) \sim (\ap)^{\nu-1} \qq \oneZap(\ap) \sim (\ap)^{1-\nu} \qq \pap\oneZap(\ap) \sim (\ap)^{-\nu}
\end{align*}
Hence we see that for interfaces with angled crests with $0<\nu<1$ we have $\oneZap(\ap) = 0$ at the singularity $\ap=0$. In the next section we will take this as the definition of singularities. From the energy $\Ecal$, we see that $\pap\oneZap \in \Ltwo$ and hence we need $\nu<\half$. Now if $z(\al,t)$ is the interface in Lagrangian coordinates, then we have
\begin{align*}
\zal = \abs{\zal}e^{i\thvar}
\end{align*}
where $\thvar $ is the angle of the interface with respect to the x-axis in Lagrangian coordinates. Hence 
\begin{align*}
\thvar = \Imag(\ln \zal) \qq \pt\thvar = \Imag\brac{{\frac{\ztal}{\zal}}}
\end{align*}
Now $\nobrac{\frac{\ztal}{\zal}} \compose \hinv = \frac{\Ztap}{\Zap}$ and we have
\begin{align*}
\Ztap = \brac{\frac{\Ztap}{\Zap}} \Zap
\end{align*}
Now from the energy $\Ecal$ we know that $\Ztap \in \Ltwo$, and for angles $0<\nu<\half$ it is clear that $\Zap$ is not in $\Ltwo$. Hence from the above equation we see that $\frac{\Ztap}{\Zap} \to 0$ as $\ap \to 0$. By changing coordinates we get $\frac{\ztal}{\zal} \to 0$ at the crest and hence the angle of the crest does not change with time. 

It is clear that the above argument is only heuristic. To make use of the above argument, we first need to show that the interface has an angled crest for $t>0$. We also need to make sense of the above equation in a pointwise sense and show that $\frac{\Ztap}{\Zap}$ is continuous at the singularity, which is implicitly used to show that $\frac{\Ztap}{\Zap} \to 0$ at the singularity. 

To prove rigorous results on the evolution of the angle of the interface, we focus mostly on getting precise evolution results for points on the boundary away from the singularities. To approach the singularity and in particular to deal with the continuity of $\frac{\Ztap}{\Zap}$ at the singularity, first observe that $\Zt$ is the boundary value of $\Fbar$ and $\oneZap$ is the boundary value of $\onePsizp$. Hence formally, the boundary value of $\onePsizp\Fzpbar$ is $\frac{\Ztap}{\Zap}$. We show that $\onePsizp\Fzpbar$ extends continuously from $\Pminus$ to $\Pminusbar$ and that $\onePsizp\Fzpbar$ vanishes at the singularities. 

Let us now understand the behavior of $\grad P$ near the singularity. Note that $\Pfrak = P \compose \Psi $ implies $(\px -i\py) P = \brac{\frac{1}{\Psizp} (\pxp - i\pyp)\Pfrak}\compose\Psiinv$ and from \eqref{eq:gradPbdry} we see that
\begin{align*}
i \frac{\Aone}{\Zap}(\ap,t) = \brac{\frac{1}{\Psizp}(\pxp - i\pyp)\Pfrak}(\ap,t) \qq \tx{ on } \partial\Pminus
\end{align*}
As $\norm[\infty]{\Aone} \leq C(\Ecalone(0)) $ and $\oneZap(\ap) \sim (\ap)^{1-\nu}$ as $\ap \to 0$, we see that $\grad P \to 0$ near the singularity. 

The issue with this argument is that the above equation is not available for singular solutions and we do not know if $\frac{1}{\Psizp} (\pxp - i\pyp)\Pfrak$ even makes sense on the boundary. To remedy this, we prove that $\frac{1}{\Psizp} (\pxp - i\pyp)\Pfrak$ extends continuously from $\Pminus$ to $\Pminusbar$ and that it vanishes at the singularities.

\medskip
\section{Main results}\label{sec:results}
\medskip

In all that follows let $T_0>0$ and let $(\F,\Psi,\Pfrak)$ be a solution to equation \eqref{eq:EulerRiem} in the time interval $[0,T_0]$ given by \thmref{thm:Wuexistence}. For $t\in [0,T_0]$ define
\begin{align*}
\tx{ Singular set } &= S(t) = \cbrac{\ap \in \Rsp \suchthat \frac{1}{\Psizp}(\ap,t) =0} \\
\tx{ Non-Singular set}  & = NS(t) = \Rsp \backslash S(t)
\end{align*}
Note that the definition makes sense as $\frac{1}{\Psizp}$ is continuous on $\Pminusbar\times[0,T_0]$. We will also identify $S(t)$ and $NS(t)$ as subsets of $\Pminusbar$ so that it is meaningful to talk of sets such as $\Pminusbar\backslash S(t)$. An important observation is that $S(t) \subset \Rsp$ is a closed set of Lebesgue measure zero. This is because it is the boundary value of a bounded holomorphic function $\frac{1}{\Psizp}$ and hence by the uniqueness theorem of F. and M. Riesz (see Theorem 17.18 in \cite{Ru87}) its zero set on the boundary is of measure zero. Hence given any $\ap\in S(t)$, there always exists a sequence $\cbrac{\al_n'}$  with $\al_n' \in NS(t)$ such that $\al_n' \to \ap$. We have a description of dynamics of this set

\begin{lem}\label{lem:propofsing}
The singular set at time $t \in [0,T_0]$ is given by $S(t) = \cbrac{h(\al,t) \in \Rsp \suchthat \al \in S(0) }$ 
\end{lem}
Note that this also implies that $NS(t) = \cbrac{h(\al,t) \in \Rsp \suchthat \al \in NS(0) }$. This lemma says that the singularities propagate via the Lagrangian flow i.e. particles at the singularities stay at the singularities. In particular the singularities are preserved in the sense that the interface doesn't smooth out or any new singularities form during the time in which $\Ecalone (t) <\infty$. This important fact is a simple consequence of the nature of the energy $\Ecalone$. We now come to our main result on the evolution of the angle of the interface at the singularities. 
\begin{thm}\label{thm:mainangle}
Let $(\F,\Psi,\Pfrak)$ be a solution in $[0,T_0]$ as given by \thmref{thm:Wuexistence}. Then
\begin{enumerate}[leftmargin =*, align=left]
\item For every fixed $t\in [0,T_0]$, the  functions  $\brac{\frac{1}{\Psizp}\F_\zp}(\cdot,t)$ and $\brac{\frac{1}{\Psizp}\Fzpbar}(\cdot,t)$ extend to continuous functions on $\Pminusbar$ with 
\begin{align*}
\brac{\frac{1}{\Psizp}\F_\zp}(\ap,t) = \brac{\frac{1}{\Psizp}\Fzpbar}(\ap,t) = 0 \quad \tx{ for all } \ap \in S(t)
\end{align*}

\item If $t\in [0,T_0]$, then we have the following formula
\begin{align*}
\frac{\Zap}{\Zapabs}(h(\al,t),t) = \frac{\Zap}{\Zapabs}(\al,0)\exp\cbrac{i\Imag \brac{\int_0^t \brac{\frac{1}{\Psizp}\Fzpbar}(h(\al,s),s)\diff s}} \quad \tx{ for all } \al \in NS(0)
\end{align*}
Moreover if $\al \in S(0)$, and if $\{\al_n\}$ is any sequence such that $\al_n \in NS(0)$ for all $n$ with $\al_n \to \al$, then 
\begin{align*}
\lim_{\al_n \to \al} \frac{\frac{\Zap}{\Zapabs}(h(\al_n,t),t)) }{\frac{\Zap}{\Zapabs}(\al_n,0)} = 1 
\end{align*}

\end{enumerate}
\end{thm}

\begin{cor}\label{cor:mainangle}
Let $(\F,\Psi,\Pfrak)$ be a solution in $[0,T_0]$ as given by \thmref{thm:Wuexistence} and assume that there exists $N\geq 1$ isolated singularities of initial interface at locations $\al_n \in S(0)$ for $1\leq n \leq N$ in conformal coordinates. Also assume that there exists unit vectors $\beta_n,\gamma_n$ for $1\leq n\leq N$, such that
\begin{align*}
\lim_{\al\to \al_n^{-}} \frac{\Zap}{\Zapabs}(\al,0) = \beta_n \q \text{ and } \lim_{\al\to \al_n^{+}} \frac{\Zap}{\Zapabs}(\al,0) = \gamma_n \q \text{ for } 1\leq i\leq N
\end{align*}
Then for all $t\in [0,T_0]$, there are N isolated singularities of the interface at locations $h(\al_n,t) \in S(t)$ in conformal coordinates. We also have
\begin{align*}
\lim_{\al\to \al_n^{-}} \frac{\Zap}{\Zapabs}(h(\al,t),t) = \beta_n \q \text{ and } \lim_{\al\to \al_n^{+}} \frac{\Zap}{\Zapabs}(h(\al,t),t) = \gamma_n \q \text{ for } 1\leq i\leq N
\end{align*}
\end{cor}
\medskip

The proof of the corollary is immediate from the theorem above. Note that in the above corollary, $\beta_n$ is the left unit tangent vector and $\gamma_n$ is the right unit tangent vector of the interface at the singularity. Hence the result says that an interface with angled crests remains angled crested and both the left and right unit tangent vectors do not change with time.

It was left open by \cite{Wu15,Wu18} whether the solutions constructed in \thmref{thm:Wuexistence} satisfy the Euler equation on the boundary. We give an affirmative answer to this question.

\begin{thm}\label{thm:gradP}
Let $(\F,\Psi,\Pfrak)$ be a solution in $[0,T_0]$ as given by \thmref{thm:Wuexistence} and fix a time $t\in [0,T_0]$. Then the  function $\brac{\frac{1}{\Psizp} (\pxp - i\pyp)\Pfrak}(\cdot,t)$ extends to continuous function on $\Pminusbar$ with 
\begin{align*}
\Zttbar(\ap,t) - i = \brac{-\onePsizp(\pxp-i\pyp) \Pfrak}(\ap,t) \qq \tx{ for all } \ap \in \Rsp
\end{align*} 
In addition we have  $\Zttbar(\ap,t) = i $ for all $\ap \in S(t)$ 

\end{thm}

\begin{cor}\label{cor:gradP}
Let $(\F,\Psi,\Pfrak)$ be a solution in $[0,T_0]$ as given by \thmref{thm:Wuexistence} such that $\Sigma(t) = \cbrac{Z = \Psi(\ap,t) \suchthat \ap\in \Rsp}$ is a Jordan curve for all $t \in [0,T_0]$, thereby giving rise to a solution of the Euler equation \eqref{eq:Euler} with $(\vboldbar,P) = (F\compose\Psiinv, \Pfrak \compose \Psiinv)$ on $\Omega(t) = \cbrac{\Psi(\zp,t) \suchthat  \zp \in \Pminus}$. Then for every fixed $t\in[0,T_0]$, the quantities $\vbold(\cdot,t), \vbold_t(\cdot,t), \grad\vbold(\cdot,t) = (\vbold_x, \vbold_y)(\cdot,t)$ and $\grad P(\cdot,t)$ extend continuously to $\Omegabar(t)$. Moreover the Euler equation 
\begin{align*}
 \mathbf{v_t + (v.\nabla)v} = -i -\nabla P  \qq\text{ on } \Omega (t)
\end{align*}
holds on $\Omegabar(t)$, with $\grad \vbold$ and $\grad P$ vanishing at the singularities $\cbrac{\Z(\ap,t)\suchthat \ap \in S(t)} \subset \Sigma(t)$. Hence the acceleration $ \mathbf{v_t + (v.\nabla)v} $ is equal to $-i$ at the singularities.
\end{cor}

Hence these results in essence say that an initial interface with angled crests stays angled crested, the particle at the tip stays at the tip, the angle doesn't change, the angle doesn't tilt and the acceleration at the tip is the one due to gravity. In particular we can now say that the singularity is "rigid". This also gives a complete description of the dynamics near the singularities as long as the energy $\Ecalone(t) $ remains finite.

\section{Proofs}\label{sec:proof}

\begin{proof}[Proof of \lemref{lem:propofsing}]
Observe that
\begin{align*}
\pt\brac{ \frac{h_\al^\ep}{z_\al^\ep} } =  \frac{h_\al^\ep}{z_\al^\ep} \brac{ \frac{h_{t\al}^\ep}{h_\al^\ep}  -  \frac{z_{t\al}^\ep}{z_\al^\ep} }
\end{align*}
Hence we have
\begin{align}\label{eq:zapep}
 \frac{h_\al^\ep}{z_\al^\ep} (\al,t) =  \frac{h_\al^\ep}{z_\al^\ep} (\al,0) \exp\cbrac{\int_0^t \brac{ \frac{h_{t\al}^\ep}{h_\al^\ep}  -  \frac{z_{t\al}^\ep}{z_\al^\ep} }(\al,s) \diff s}
\end{align}
Hence from \eqref{ineq:quantE} we see that there exists $c_1,c_2 >0$ depending only on $\Ecalone(0)$ and $T_0$ such that
\begin{align*}
c_1\abs{ \frac{h_\al^\ep}{z_\al^\ep} (\al,0)} \leq \abs{ \frac{h_\al^\ep}{z_\al^\ep} (\al,t)} \leq c_2\abs{ \frac{h_\al^\ep}{z_\al^\ep} (\al,0)} \qq \tx{ for all } \al \in \Rsp, t\in [0,T_0]
\end{align*}
Recall that $\frac{\halep}{\zalep} = \oneZapep\compose \hep$ and $\hep(\al,0) = \al$. Hence
\begin{align*}
c_1\abs{\oneZapep(\al,0)} \leq \abs{\oneZapep(\hep(\al,t),t)} \leq c_2\abs{\oneZapep(\al,0)}
\end{align*}
As $\oneZapep(\ap,t) = \onePsizpep(\ap,t)$, letting $\ep \to 0$ we see that for all $\al \in \Rsp$ and $t\in [0,T_0]$ we have
\begin{align*}
c_1\abs{\onePsizp(\al,0)} \leq \abs{\onePsizp(\h(\al,t),t)} \leq c_2\abs{\onePsizp(\al,0)}
\end{align*}
which proves the lemma.
\end{proof}

\begin{lem}\label{lem:ctsextension}
Let $g:\Rsp\to \Csp$ be a continuous function and let $C>0$ be a constant. Let $A\subset \Rsp$ be a set of full measure and let $f:A\to \Csp$ be such that
\begin{align*}
\abs{f(x)-f(y)} \leq C\abs{g(x)-g(y)} \qq \tx{ for all } x,y\in A
\end{align*}
Then there exists a unique continuous function $\widetilde{f}: \Rsp\to\Csp$  such that $\widetilde{f}\vert_A = f$
\end{lem}
\begin{proof}
Clearly the constant $C$ can be absorbed into the function $g$ and so without loss of generality we assume $C=1$. As $A$ is a set of full measure,  $A$ is dense in $\Rsp$. Fix $x\in\Rsp$ and choose a sequence $(x_n) $ with $x_n\in A$ and $x_n \to x$. Now
\begin{align*}
\abs{f(x_n)-f(x_m)} \leq \abs{g(x_n)-g(x_m)}
\end{align*}
and hence $\cbrac{f(x_n)}$ is a Cauchy sequence. Define $\widetilde{f}(x) = \lim_{n\to\infty} f(x_n)$. We easily see that $\widetilde{f} $ is well defined as if $(x_n') $ is another sequence with $x_n' \in A$ and $x_n' \to x$, then 
\begin{align*}
\abs{f(x_n)-f(x_n')} \leq \abs{g(x_n)-g(x_n')} \to 0 \quad \tx{ as } n\to \infty
\end{align*}
From this we also see that $\widetilde{f}\vert_A = f$ and 
\begin{align*}
\abs*{\widetilde{f}(x)-\widetilde{f}(y)} \leq \abs{g(x)-g(y)} \qq \tx{ for all } x,y\in \Rsp
\end{align*}
Hence $\widetilde{f}$ is a continuous function and $\widetilde{f}$ is unique, as a continuous function is determined by its values on a dense set.
\end{proof}

\begin{lem}\label{lem:unifconvPminusbar}
Let $f_n:\Rsp \to \Csp$ be a sequence of smooth functions. Let $1<p\leq \infty$ and suppose that there exists a constant $C >0$ independent of $n$, such that $\norm[\Linfty]{f_n} + \norm[L^p]{\px f_n} \leq C$. Then there exists a continuous and bounded function $f:\Rsp \to \Csp$ and a subsequence $\cbrac{f_{n_j}}$ such that $K_\yp\conv f_{n_j} \to K_\yp\conv f$ uniformly on compact subsets of $\Pminusbar$.
\end{lem}
\begin{proof}
This is an easy consequence of the Arzela-Ascoli theorem and the fact that if $ f_{n_j} \to  f$ uniformly on compact subsets of $\Rsp$, then $K_\yp\conv f_{n_j} \to K_\yp\conv f$ uniformly on compact subsets of $\Pminusbar$.
\end{proof}

\begin{lem}\label{lem:Linftybound}
Let $(\F,\Psi,\Pfrak)$ be a solution in $[0,T_0]$ as given by \thmref{thm:Wuexistence}. Then for each $t\in [0,T_0]$ we have
\begin{align*}
\sup_{\yp<0}\norm[\Linfty(\Rsp,\diff \xp)]{\frac{1}{\Psizp}\F_\zp(\xp+i\yp,t)} + \sup_{y'<0}\norm[\Linfty(\Rsp,\diff \xp)]{\partial_{z'}\brac{\frac{1}{\Psizp^2}}(\xp+i\yp,t)} \leq C(\Ecal_1(t))
\end{align*}
\end{lem}
\begin{proof}
We note that for any $t\in[0,T_0]$ we have
\begin{align*}
 \sup_{\yp <0}\norm[\Ltwo(\Rsp,\diff \xp)]{\frac{1}{\Psizp(\xp+i\yp,t)} - 1}^2 + \sup_{\yp<0}\norm[\Ltwo(\Rsp, \diff \xp)]{\partial_\zp \brac{\frac{1}{\Psizp}}(\xp+i\yp,t)}^2 < \infty
\end{align*}
Hence for any fixed $\yp<0$ we have $\onePsizp(\xp + i\yp,t) \to 1$ as $\abs{\xp} \to \infty$ . Hence from the energy $\Ecalone(t)$ we see that for any fixed $\yp<0$, there exists $N>0$ large enough so that
\begin{align*}
\norm[\Ltwo(\Rsp, \diff \xp)]{\frac{1}{\Psizp}\Fzp(\xp+i\yp,t)}^2+ \norm[\Ltwo(\Rsp \backslash [-N,N], \diff \xp)]{\pzp\brac{\frac{1}{\Psizp}\Fzp}(\xp+i\yp,t)}^2 < \infty
\end{align*}
Hence for any fixed $\yp<0$, we see that $\frac{1}{\Psizp}\Fzp(\xp+i\yp,t) \to 0$ as $\abs{\xp} \to \infty$. Now observe that for $\yp<0$ we have 
\begin{align}
& \abs{\brac{\frac{1}{\Psizp}\F_\zp}^2(\al_2+i\yp,t) - \brac{\frac{1}{\Psizp}\F_\zp}^2(\al_1+i\yp,t)} \\
&\leq 2\int_{\al_1}^{\al_2} \abs{\F_\zp}\abs{\frac{1}{\Psizp}\pzp\brac{\frac{1}{\Psizp}\F_\zp}}(s+i\yp,t) \diff s
\end{align}
Hence we have $\dis \sup_{\yp<0}\norm[\Linfty(\Rsp,\diff \xp)]{\frac{1}{\Psizp}\F_\zp(\xp+i\yp,t)} \leq C(\Ecal_1(t))$. A similar argument shows that $\dis \sup_{y'<0}\norm[\Linfty(\Rsp,\diff s)]{\partial_{z'}\brac{\frac{1}{\Psizp^2}}(\xp+i\yp,t)} \leq C(\Ecal_1(t))$
\end{proof}

From the heuristics, we know that the energy $\Ecalone(t) < \infty$ allows angles $\nu\pi$ for $0<\nu<\half$. Also for such interfaces $\Zap (\ap) \sim (\ap)^{\nu-1}$ near $\ap=0$, and hence $\Zap$ does not belong in $\Ltwo$. We now give a more general argument of this heuristic which will be required to prove \thmref{thm:mainangle}. 
\begin{lem}\label{lem:Psiblowup}
Let $(\F,\Psi,\Pfrak)$ be a solution in $[0,T_0]$ as given by \thmref{thm:Wuexistence}. Fix a $t \in [0,T_0]$ and let $\ap \in S(t)$. Then for all $\delta>0$ we have
\begin{align*}
\lim_{\yp \to 0^-}\int_{\al'-\delta}^{\al'+\delta}\abs{\Psizp}^2(s + iy',t)\diff s = \infty 
\end{align*}
\begin{proof}
Without loss of generality we assume $\al'=0$. We have for $y'<0$
\begin{align*}
\frac{1}{\Psizp^2}(\xp+i\yp,t) = \frac{1}{\Psizp^2}(0+i\yp,t) + \int_0^\xp \partial_{z'}\brac{\frac{1}{\Psizp^2}}(s+iy',t)\diff s
\end{align*}
From \lemref{lem:Linftybound} we set $\dis C = \sup_{y'<0}\norm[\Linfty(\Rsp,\diff s)]{\partial_{z'}\brac{\frac{1}{\Psizp^2}}(s+i\yp,t)} <\infty$. Hence we have
\begin{align*}
\frac{1}{\abs{\Psizp}^2}(\xp+ i\yp,t) \leq \frac{1}{\abs{\Psizp}^2}(0+i\yp,t) + C\abs{\xp} \quad \tx{ for all } \xp\in\Rsp, \yp <0
\end{align*}
From this we see that for $y'<0$
\begin{align*}
\int_{-\delta}^{\delta}\abs{\Psizp}^2(s+i\yp,t)\diff s \geq \int_{-\delta}^{\delta}\frac{1}{\nobrac{\frac{1}{\abs{\Psizp}^2}(0+i\yp,t)} + C\abs{s}} \diff s
\end{align*}
As $\frac{1}{\Psizp}(\cdot,t)$ is continuous on $\Pminusbar$ and by assumption $\frac{1}{\Psizp}(0,t) = 0$, we have that $\frac{1}{\Psizp^2}(0+iy',t) \to 0$ as $y'\to 0^-$ proving the lemma.
\end{proof}
\end{lem}

\medskip
\begin{proof}[Proof of \thmref{thm:mainangle}]
In steps 1-3 we fix time $t\in [0,T_0]$ and we prove that $\onePsizp\Fzp(\cdot,t)$ extends continuously to $\Pminusbar$ and that it vanishes on $S(t)$. From this, the result for $\onePsizp\Fzpbar(\cdot,t)$ follows easily. In steps 4-6 we prove the evolution formula of the unit tangent vector of the interface. 
\medskip

\noindent \textbf{Step 1}: 
As $\onePsizp\Fzp(\cdot,t)$ is a bounded holomorphic function from \lemref{lem:Linftybound}, by Fatou's theorem there exists $f\in \Linfty(\Rsp)$ and a set $A\subset \Rsp$ of full measure such that for $\yp<0$ we have
\begin{align*}
\brac{\frac{1}{\Psizp}\F_\zp} (\cdot + i\yp,t) &= K_\yp \conv f \\
 \tx{ and }\quad  \brac{\frac{1}{\Psizp}\F_\zp} (\ap+i\yp,t) &\to f(\ap) \quad \tx{ as } \yp\to 0 \quad \tx{ for all } \ap\in A 
\end{align*}
Also as $\dis \sup_{\yp<0}\norm[\Ltwo(\Rsp,\diff \xp)]{\F_\zp(\xp+i\yp,t)} < C(\Ecalone(t))$ and $\F_\zp(\cdot,t)$ is holomorphic, there exists $g_1 \in \Ltwo(\Rsp)$ such that 
\begin{align*}
\F_\zp(\cdot+i\yp,t) \to g_1\quad  \tx{ in } \Ltwo \quad \tx{ as } \yp\to 0 
\end{align*}
Similarly there exists $g_2 \in \Ltwo$ such that 
\begin{align*}
\frac{1}{\Psizp}\pzp\brac{\frac{1}{\Psizp }\F_\zp}(\cdot+i\yp,t) \to g_2 \quad \tx{ in } \Ltwo \quad \tx{ as } \yp \to 0
\end{align*}
Hence we see that
\begin{align*}
\abs{\F_\zp}\abs{\frac{1}{\Psizp}\pzp\brac{\frac{1}{\Psizp }\F_\zp}}(\cdot+i\yp,t) \to \abs{g_1}\abs{g_2} \quad \tx{ in } \Lone
 \quad \tx{ as } \yp\to 0 
\end{align*}
Define the function $h:\Rsp\to\Rsp$
\begin{align*}
h(\al) = \int_0^\al \abs{g_1}\abs{g_2}(s) \diff s
\end{align*}
Clearly $h$ is a continuous function on $\Rsp$. Now observe that for $\yp<0$ we have 
\begin{align*}
 \abs{\brac{\frac{1}{\Psizp}\F_\zp}^2(\al_2+i\yp,t) - \brac{\frac{1}{\Psizp}\F_\zp}^2(\al_1+i\yp,t)} \leq 2\int_{\al_1}^{\al_2} \abs{\F_\zp}\abs{\frac{1}{\Psizp}\pzp\brac{\frac{1}{\Psizp}\F_\zp}}(s+i\yp,t) \diff s
\end{align*}
Letting $\yp\to0$ we obtain
\begin{align*}
\abs{f^2(\al_2) - f^2(\al_1)} \leq 2\abs{h(\al_2) - h(\al_1)} \quad \tx{ for all } \al_1,\al_2 \in A
\end{align*}
Hence by \lemref{lem:ctsextension} there exists a continuous function $f_2:\Rsp\to\Csp$ such that $f_2\vert_{A} = f^2$. Also observe that
\begin{align*}
& \abs{\brac{\frac{1}{\Psizp}\F_\zp}^{ 3}(\al_2+i\yp,t) - \brac{\frac{1}{\Psizp}\F_\zp}^{3}(\al_1+i\yp,t)}  \\
&  \leq C(\Ecal_1(t)) \int_{\al_1}^{\al_2} \abs{\F_\zp}\abs{\frac{1}{\Psizp}\pzp\brac{\frac{1}{\Psizp}\F_\zp}}(s+i\yp,t) \diff s
\end{align*}
Hence via the same argument there exists a continuous function $f_3:\Rsp\to\Csp$ such that $f_3\vert_{A} = f^3$.

\medskip

\noindent \textbf{Step 2}: Define the function $\ftil:\Rsp \to \Csp$ 
\begin{align*}
\ftil(\ap) = 
\begin{cases} 
(f_3/f_2)(\ap) & \text{if } f_2(\ap) \neq 0 \\
0       & \text{otherwise } 
\end{cases}
\end{align*}
We claim that $\ftil $ is a continuous function on $\Rsp$ and $\ftil\vert_A = f$. 

First note that both $f_3$ and $f_2$ are continuous. Fix $\ap\in\Rsp$ and observe that if $f_2(\ap) \neq 0$, then $\ftil $ is continuous at $\ap$. Hence we need to prove the continuity of $\ftil$ at $\ap$ where $f_2(\ap) =0$. Define the function $\fabs:\Rsp \to \Rsp$ by $\fabs = \sqrt{\abs{f_2}}$. Observe that $\fabs$ is a continuous function on $\Rsp$ and that $\abs{f_3}(\ap) = \fabs^3(\ap)$ for all $\ap \in A$. As $A$ is a set of full measure and both $\abs{f_3}$ and $\fabs^3$ are continuous functions, we have $\abs{f_3}(\ap) = \fabs^3(\ap)$ for all $\ap \in \Rsp$. Hence we now see that $\abs*{\ftil(\ap)} \leq \fabs(\ap)$ for all $\ap\in\Rsp$ and this inequality is enough to prove continuity of $\ftil$ at all $\ap$ where $f_2(\ap)=0$. Hence $\ftil$ is continuous on $\Rsp$.

Recall that $f_2\vert_A = f^2$ and $f_3\vert_A = f^3$. If $\ap \in A$ and $f(\ap) \neq 0$, then $\ftil(\ap) = (f_3/f_2)(\ap) = f(\ap)$. If $\ap \in A$ and $f(\ap) = 0$, then we see that $f_2(\ap)=0$ and hence $\ftil(\ap) =0$. Hence $\ftil\vert_A = f$.

\medskip
\noindent \textbf{Step 3}: As $\ftil\vert_A = f$ and $A$ is a set of full measure we now have
\begin{align*}
\frac{1}{\Psizp}\F_\zp (\cdot+i\yp,t) &= K_\yp\conv\ftil \quad \tx{ for all } \yp<0
\end{align*}
As $f$ is bounded, we see that $\ftil $ is a continuous and bounded function, and hence $\frac{1}{\Psizp}\F_\zp(\cdot,t)$ extends continuously to $\Pminusbar$. Now let $\ap\in S(t)$ be a singular point. We proceed via contradiction and assume that $\frac{1}{\Psizp}\F_\zp(\ap,t) = c \neq 0$. Hence there exists $c_1,c_2, \delta>0$ so that 
\begin{align*}
0<c_1 \leq \abs{\frac{1}{\Psizp}\F_\zp}(s+i\yp,t) \leq c_2 < \infty \quad \tx{ for all } s\in(\ap-\delta,\ap+\delta) \tx{ and } -\delta\leq \yp<0
\end{align*}
Observe that for $\yp<0$ we have
\begin{align*}
\F_\zp(\ap+i\yp,t) = \Psizp(\ap+i\yp,t)\brac{\frac{1}{\Psizp}\F_\zp(\ap+i\yp,t)}
\end{align*}
and hence we obtain 
\begin{align*}
\abs{\F_\zp(s+i\yp,t)} \geq c_1\abs{\Psizp(s+i\yp,t)} \quad \tx{ for all } s\in(\ap-\delta,\ap+\delta) \tx{ and } -\delta\leq \yp<0
\end{align*}
By integrating we get
\begin{align*}
\int_{\ap-\delta}^{\ap+\delta} \abs{\F_\zp(s+i\yp,t)}^2 \diff s \geq c_1^2 \int_{\ap-\delta}^{\ap+\delta} \abs{\Psizp(s+i\yp,t)}^2 \diff s  \quad \tx{ for all } -\delta\leq \yp<0
\end{align*}
Letting $\yp\to 0^-$ and using \lemref{lem:Psiblowup} we see that
\begin{align*}
\lim_{\yp\to 0^-} \int_{\ap-\delta}^{\ap+\delta} \abs{\F_\zp(s+i\yp,t)}^2 \diff s =  \infty
\end{align*}
which contradicts the finiteness of the energy, $\dis \sup_{\yp<0} \norm[\Ltwo(\Rsp,\diff \xp)]{\F_\zp(\xp+i\yp,t)}^2 \leq \Ecalone(t) < \infty$

We have proven the result for $\frac{1}{\Psizp}\F_\zp$ and we now need to prove the result for $\frac{1}{\Psizp}\Fzpbar$. We observe that $\Psizp$ extends continuously to $\Pminusbar\backslash S(t)$ and hence the functions $\F_\zp$ and $\frac{1}{\Psizp}\Fzpbar$ extend continuously to  $\Pminusbar\backslash S(t)$. As $\abs{\frac{1}{\Psizp}\Fzpbar} = \abs{\frac{1}{\Psizp}\F_\zp}$ on $\Pminus$ and as $\frac{1}{\Psizp}\F_\zp$ extends continuously to $\Pminusbar$ and vanishes on $S(t)$, this forces $\frac{1}{\Psizp}\Fzpbar$ to extend continuously to $\Pminusbar$ and $\frac{1}{\Psizp}\Fzpbar(\ap,0) = 0$ for all $\ap \in S(t)$.

\medskip
\noindent \textbf{Step 4}: Observe that $\frac{1}{(\Zapep)^2}\Ztapepbar$ is the boundary value of the holomorphic function $\frac{1}{(\Psizpep)^2}\Fzpep$. From \eqref{ineq:quantE} and \eqref{ineq:quantE2} we see that for all $0<\ep\leq1$, we have the estimates
\begin{align*}
\norm[\infty]{\frac{1}{(\Zapep)^2}\Ztapepbar(\cdot,t)} \leq \norm[\infty]{\frac{1}{\Zapep}(\cdot,t)}\norm[\infty]{\frac{1}{\Zapep}{\Ztapepbar}(\cdot,t)} \leq C(c_0,\Ecalone(0))
\end{align*}
and also
\begin{align*}
\norm[2]{\pap\brac*[\bigg]{\frac{1}{(\Zapep)^2}\Ztapepbar}(\cdot,t)} \leq C(\Ecalone(0))
\end{align*}
Hence from \lemref{lem:unifconvPminusbar}, we see that there exists a sequence $\ep_j = \ep$ such that as $\ep\to 0$, the functions $\frac{1}{(\Psizpep)^2}\Fzpep(\cdot,t)$ converge on compact subsets of $\Pminusbar$ to a continuous and bounded function. But from the proof of \thmref{thm:Wuexistence} we already know that as $\ep\to 0$
\begin{align*}
\onePsizpep(\cdot,t) \Ra \onePsizp(\cdot,t) \qq \Fzpep(\cdot,t) \Ra \Fzp(\cdot,t) \qq \tx{ on } \Pminus
\end{align*}
Hence $\frac{1}{(\Psizp)^2}\Fzp(\cdot,t)$ extends continuously to $\Pminusbar$ and 
\begin{align*}
\frac{1}{(\Psizpep)^2}\Fzpep(\cdot,t) \Ra \frac{1}{(\Psizp)^2}\Fzp(\cdot,t) \qq \tx{ on } \Pminusbar
\end{align*}
As $\onePsizpep(\cdot,t) \Ra \onePsizp(\cdot,t) $ on $\Pminusbar$ and by the definition of $S(t)$ we have
\begin{align*}
\Psizpep(\cdot,t) \Ra \Psizp(\cdot,t) \qq \tx{ on } \Pminusbar\backslash S(t)
\end{align*}
Using this we see that as $\ep\to0$ we have
\begin{align*}
\frac{1}{\Psizpep}\Fzpep(\cdot,t) \Ra \frac{1}{\Psizp}\Fzp(\cdot,t) \qq \frac{1}{\Psizpep}\Fzpepbar(\cdot,t) \Ra \frac{1}{\Psizp}\Fzpbar(\cdot,t) \qq \tx{ on } \Pminusbar\backslash S(t)
\end{align*}

\medskip
\noindent \textbf{Step 5}:
From \eqref{eq:zapep} we see that
\begin{align*}
 \frac{h_\al^\ep}{z_\al^\ep} (\al,t) =  \frac{h_\al^\ep}{z_\al^\ep} (\al,0) \exp\cbrac{\int_0^t \brac{ \frac{h_{t\al}^\ep}{h_\al^\ep}  -  \frac{z_{t\al}^\ep}{z_\al^\ep} }(\al,s) \diff s}
\end{align*}
and hence by inverting and by using $\frac{\zal^\ep}{\hal^\ep}  = \Zap^\ep \compose h^\ep$ we get
\begin{align*}
\Zap^\ep(h^\ep(\al,t),t) = \Zap^\ep(\al,0)\exp\cbrac{\int_0^t \brac*[\Big]{\frac{\Ztap^\ep}{\Zap^\ep} - \bvarap^\ep}(h^\ep(\al,s),s)\diff s } 
\end{align*}
We see that $\frac{\Ztapep}{\Zapep}(\ap,t) = \brac{\onePsizpep\Fzpepbar}(\ap,t)$, and from this we obtain
\begin{align*}
\frac{\Zap^\ep}{\abs*{\Zap^\ep}}(h^\ep(\al,t),t) = \frac{\Zap^\ep}{\abs*{\Zap^\ep}}(\al,0)\exp\cbrac{i\Imag \brac{\int_0^t \brac{\onePsizpep\Fzpepbar}(h^\ep(\al,s),s)\diff s}} \quad \tx{ for all } \al \in \Rsp
\end{align*}
Recall that 
\begin{align*}
\frac{1}{\Psizp^\ep} \Rightarrow \frac{1}{\Psizp} \quad \tx{ on } \Pminusbar\times [0,T_0] \qq \tx{ and } h^\ep \Rightarrow h \quad \text{ on } \Rsp\times[0,T_0]
\end{align*}
and we have shown that for all $t\in[0,T_0]$
\begin{align*}
\Psizpep(\cdot,t) \Ra \Psizp(\cdot,t) \qq \frac{1}{\Psizpep}\Fzpepbar(\cdot,t) \Ra \frac{1}{\Psizp}\Fzpbar(\cdot,t) \qq \tx{ on } \Pminusbar\backslash S(t)
\end{align*}
Hence letting $\ep \to 0$, using dominated convergence theorem we see that for any $t\in[0,T_0]$
\begin{align*}
\frac{\Zap}{\Zapabs}(h(\al,t),t) = \frac{\Zap}{\Zapabs}(\al,0)\exp\cbrac{i\Imag \brac{\int_0^t \brac{\frac{1}{\Psizp}\Fzpbar}(h(\al,s),s)\diff s}} \quad \tx{ for all } \al \in NS(0)
\end{align*}
proving the evolution formula.

\medskip
\noindent \textbf{Step 6}:
Now if $\al_n \in NS(0)$, using the evolution formula we see that
\begin{align*}
\frac{\Zap}{\Zapabs}(h(\al_n,t),t) = \frac{\Zap}{\Zapabs}(\al_n,0)\exp\cbrac{i\Imag \brac{\int_0^t \brac{\frac{1}{\Psizp}\Fzpbar}(h(\al_n,s),s)\diff s}} 
\end{align*}
Now from step 3 we know that $\brac{\frac{1}{\Psizp}\Fzpbar}(\zp,t)$ is continuous in $\zp$ on $\Pminusbar$, and hence by dominated convergence theorem we see that 
\begin{align*}
\lim_{\al_n \to \al}\frac{\frac{\Zap}{\Zapabs}(h(\al_n,t),t)) }{\frac{\Zap}{\Zapabs}(\al_n,0)} = \exp\cbrac{i\Imag \brac{\int_0^t \brac{\frac{1}{\Psizp}\Fzpbar}(h(\al,s),s)\diff s}} 
\end{align*}
But from step 3 we also know that if $\al\in S(0)$, then $\brac{\frac{1}{\Psizp}\Fzpbar}(h(\al,s),s) = 0$ for all $s\in[0,T_0]$. Hence the result is proved.
\end{proof}

\medskip
\begin{proof}[Proof of \thmref{thm:gradP}]
We prove this theorem in steps.

\medskip
\noindent \textbf{Step 1}: From \eqref{form:Pfrak} we know that 
\begin{align*}
\Pfrak(\zp,t) = -\half\abs{\F(\zp,t)}^2 - y + \half K_\yp \conv (\abs{\Ztbar}^2)(\xp,t) \quad \tx{ on } \Pminus\times[0,T_0]
\end{align*}
As $2\pzp = (\pxp - i\pyp)$ we obtain
\begin{align*}
-\onePsizp(\pxp-i\pyp) \Pfrak =  \Fbar\frac{\Fzp}{\Psizpep} - \frac{i}{\Psizpep} - \onePsizp\pzp K_\yp\conv (\abs{\Ztbar}^2) \quad \tx{ on } \Pminus\times[0,T_0]
\end{align*}
Define  $\G:\Pminus\times[0,T_0] \to \Csp$ by
\begin{align*}
\G  = \onePsizp \pzp K_\yp\conv(\abs{\Ztbar}^2)
\end{align*}
Observe that as $K_\yp\conv(\abs{\Ztbar}^2)$ is a harmonic function, $\G$ is holomorphic. We obtain
\begin{align}\label{form:gradPfrak}
-\onePsizp(\pxp-i\pyp) \Pfrak =  \Fbar\frac{\Fzp}{\Psizp} - \frac{i}{\Psizp} - \G \quad \tx{ on } \Pminus\times [0,T_0]
\end{align} 
For $0<\ep\leq 1$ define $\Gep:\Pminusbar\times[0,T_0] \to \Csp$ by
\begin{align*}
\Gep & = \onePsizpep \pzp K_\yp\conv(\abs{\Ztbarep}^2) 
\end{align*}
hence $\Gep$ is holomorphic and we have
\begin{align}\label{eq:gradPep}
-\onePsizpep(\pxp-i\pyp) \Pfrakep =  \Fepbar\frac{\Fzpep}{\Psizpep} - \frac{i}{\Psizpep} - \Gep \quad \tx{ on } \Pminus\times[0,T_0]
\end{align}
From the proof of \thmref{thm:Wuexistence} we know that
\begin{align*}
(\pxp-i\pyp)\Pfrakep \Rightarrow (\pxp-i\pyp)\Pfrak \quad \tx{ on } \Pminus\times [0,T_0] \\
\Fep \Rightarrow \F \qq \Fzpep \Rightarrow \Fzp \qq \onePsizpep \Rightarrow \onePsizp \qq \tx{ on } \Pminus\times [0,T_0] 
\end{align*}
Hence we see that 
\begin{align*}
\Gep \Rightarrow \G \qq\tx{ on } \Pminus\times [0,T_0]
\end{align*}

\medskip
\noindent \textbf{Step 2}: Define the boundary value of $\Gep$ as $\gep:\Rsp\times[0,T_0] \to \Csp$ given by
\begin{align*}
\gep(\ap,t) = \Gep(\ap,t)
\end{align*}
Writing \eqref{eq:gradPep} on the boundary and using \eqref{eq:gradPbdry} we get
\begin{align*}
-i\frac{\Aone^\ep}{\Zapep} =  \Ztep\frac{\Ztapepbar}{\Zapep} -\frac{i}{\Zapep} - \gep
\end{align*}
Hence we see that
\begin{align}\label{eq:gep}
\gep = i\frac{\Aone^\ep}{\Zapep} + \Ztep\frac{\Ztapepbar}{\Zapep} - \frac{i}{\Zapep}
\end{align}
We obtain the following estimate from \eqref{ineq:quantE} and \eqref{ineq:quantE2}
\begin{align*}
\norm[\infty]{\gep(\cdot,t)} \leq \brac{\norm[\infty]{\Aone^\ep(\cdot,t)} + 1}\norm[\infty]{\oneZapep(\cdot,t)} + \norm[\infty]{\Ztep(\cdot,t)}\norm[\infty]{\frac{\Ztapepbar}{\Zapep}(\cdot,t)} \leq C(c_0,\Ecalone(0)) 
\end{align*}
As $\gep$ is the boundary value of the holomorphic function $\Gep$, we see that for all $0<\ep\leq 1$
\begin{align*}
\sup_{\yp<0}\norm[\Linfty(\Rsp,\diff \xp)]{\Gep(\xp+i\yp,t)} \leq C(c_0,\Ecalone(0)) 
\end{align*}
As $\Gep(\zp,t) \Ra \G(\zp,t)$ on $\Pminus\times[0,T_0]$, we see that $G(\cdot,t)$ is a bounded holomorphic function on $\Pminus$.

\medskip
\noindent \textbf{Step 3}: For $0<\ep\leq 1$ we have the estimates
\begin{align*}
\norm[\infty]{\frac{\gep}{\Zapep}(\cdot,t)} \leq \norm[\infty]{\gep(\cdot,t)}\norm[\infty]{\frac{1}{\Zapep}(\cdot,t)} \leq C(c_0,\Ecalone(0))
\end{align*}
and also 
\begin{align*}
& \norm[2]{\pap\brac*[\bigg]{\frac{\gep}{\Zapep}}(\cdot,t)} \\
& \leq 2\norm[2]{\brac*[\bigg]{\pap\frac{1}{\Zapep}}(\cdot,t)}\norm[\infty]{\oneZapep(\cdot,t)}\brac{\norm[\infty]{\Aone^\ep(\cdot,t)} + 1} + \norm[\infty]{\oneZapep(\cdot,t)}\norm[2]{\oneZapep\pap\Aone^\ep(\cdot,t)} \\
& \quad + \norm[2]{\Ztapep(\cdot,t)}\norm[\infty]{\oneZapep(\cdot,t)}\norm[\infty]{\frac{\Ztapepbar}{\Zapep}(\cdot,t)} + \norm[\infty]{\Ztep(\cdot,t)}\norm[2]{\pap\brac*[\bigg]{\frac{\Ztapepbar}{(\Zapep)^2}}(\cdot,t)} \\
& \quad \leq C(c_0,\Ecalone(0)) 
\end{align*}
As $\frac{\gep}{\Zapep}(\cdot,t)$ is the boundary value of the holomorphic function $\frac{\Gep}{\Psizpep}(\cdot,t)$, using \lemref{lem:unifconvPminusbar} we see that there is sequence $\ep_j = \ep$ such that as $\ep \to 0$, the functions $\frac{\Gep}{\Psizpep}(\cdot,t)$ converge on compact subsets of $\Pminusbar$ to a continuous and bounded function. But we already know that as $\ep\to 0$
\begin{align*}
\onePsizpep(\cdot,t) \Ra \onePsizp(\cdot,t) \qq \Gep(\cdot,t) \Ra \G(\cdot,t) \qq \tx{ on } \Pminus
\end{align*}
Hence $\frac{\G}{\Psizp}(\cdot,t)$ extends continuously to $\Pminusbar$ and 
\begin{align*}
\frac{\Gep}{\Psizpep}(\cdot,t)  \Ra \frac{\G}{\Psizp}(\cdot,t) \qq \tx{ on } \Pminusbar
\end{align*}
As $\Psizpep(\cdot,t) \Ra \Psizp(\cdot,t)$ on $\Pminusbar\backslash S(t)$, we see that $G(\cdot,t)$ extends continuously to $\Pminusbar\backslash S(t)$ and 
\begin{align*}
\Gep(\cdot,t) \Ra \G(\cdot,t) \qq \tx{ on } \Pminusbar\backslash S(t)
\end{align*}

\medskip
\noindent \textbf{Step 4}: We know form the proof of \thmref{thm:Wuexistence} that $\Zttep \Ra \Ztt$ on $\Rsp\times[0,T_0]$ and $\Ztt$ is a continuous and bounded function. From \eqref{form:Zttbar} we have
\begin{align*}
\Zttbarep(\ap,t) - i = -i\frac{\Aone^\ep}{\Zapep}(\ap,t)
\end{align*}
Now recall that $\abs{\Aone^\ep}(\ap,t) \leq C$ for all $\ap\in\Rsp$ and $t \in [0,T_0]$, where $C = C(\Ecalone(0)) > 0$. Hence we see that 
\begin{align*}
\abs{\Zttbarep(\ap,t) - i } \leq \frac{C}{\abs{\Psizpep}}(\ap,t)
\end{align*}
Letting $\ep \to 0$, we see that
\begin{align*}
\abs{\Zttbar(\ap,t) - i } \leq \frac{C}{\abs{\Psizp}}(\ap,t)
\end{align*}
Hence $\Zttbar(\ap,t) = i$ for all $\ap\in S(t)$. Now from \eqref{eq:gep} and \eqref{form:Zttbar} we have
\begin{align*}
\gep = -(\Zttbarep - i) + \Ztep\frac{\Ztapepbar}{\Zapep} - \frac{i}{\Zapep}
\end{align*}
As $G(\cdot,t)$ extends continuously to $\Pminusbar\backslash S(t)$, we define $g(\cdot,t): \Rsp \to \Csp$ as 
\begin{align*}
g(\ap,t) = 
\begin{cases}
\G(\ap,t) \qq &\tx{ if } \ap \in NS(t) \\
0  \qq &\tx{ if } \ap \in S(t)
\end{cases}
\end{align*}
As $NS(t)$ is a set of full measure and $G(\cdot,t)$ is a bounded holomorphic function, we see that 
\begin{align*}
G(\ap,t) = (K_\yp \conv g)(\ap,t) \qq \tx{ for all } (\ap,t) \in \Pminus\times[0,T_0] 
\end{align*}
Now we already know that $g(\cdot,t)$ is continuous on $NS(t)$ and $\gep(\cdot,t) \Ra \g(\cdot,t) $ on $NS(t)$. Hence using step 4 of the proof of \thmref{thm:mainangle}, by letting $\ep \to 0$ we obtain
\begin{align}\label{form:g}
g(\ap,t) = -(\Zttbar -i) + \Fbar(\ap,t)\brac{\onePsizp\Fzp}(\ap,t) - \frac{i}{\Psizp}(\ap,t) \qq \tx{ for all } \ap \in NS(t)
\end{align}
Now note that $\Fbar(\ap,t)$ is bounded and hence by using \thmref{thm:mainangle} we see that $g(\cdot,t)$ is also continuous on $S(t)$ with $g(\ap,t) = 0$ for all $\ap \in S(t)$. Hence $g(\cdot,t)$ is continuous on $\Rsp$ and hence $G(\cdot,t)$ extends continuously to $\Pminusbar$ with $G(\ap,t) = 0$ for all $\ap\in S(t)$. Now using the formula from \eqref{form:gradPfrak} namely
\begin{align*}
-\onePsizp(\pxp-i\pyp) \Pfrak =  \Fbar\frac{\Fzp}{\Psizp} - \frac{i}{\Psizp} - \G \quad \tx{ on } \Pminus\times[0,T_0]
\end{align*} 
we see that $\brac{\onePsizp(\pxp-i\pyp) \Pfrak}(\cdot,t)$ extends continuously to $\Pminusbar$ and $\brac{\onePsizp(\pxp-i\pyp) \Pfrak}(\ap,t) = 0$ for all $\ap \in S(t)$. In addition to this, observe from \eqref{form:g} and \eqref{form:gradPfrak} that we have
\begin{align*}
\Zttbar(\ap,t) - i = \brac{-\onePsizp(\pxp-i\pyp) \Pfrak}(\ap,t) \qq \tx{ for all } \ap \in NS(t)
\end{align*}
But as $\Zttbar(\ap,t) = i$ for all $\ap \in S(t)$ and $\brac{\onePsizp(\pxp-i\pyp) \Pfrak}(\ap,t) = 0$ for all $\ap \in S(t)$, we see that 
\begin{align*}
\Zttbar(\ap,t) - i = \brac{-\onePsizp(\pxp-i\pyp) \Pfrak}(\ap,t) \qq \tx{ for all } \ap \in \Rsp
\end{align*} 
\end{proof}

\begin{proof}[Proof of \corref{cor:gradP}]
The statements are all easily proven 
\begin{enumerate}
\item As $\vbold = \Fbar \compose \Psiinv$ and as $\Fbar(\cdot,t)$ extends continuously to $\Pminusbar$, we see that $\vbold(\cdot,t)$ extends continuously to $\Omegabar$.

\item  Observe that as $\vboldbar(\cdot,t)$ is holomorphic on $\Omega(t)$ and $\vboldbar = \F\compose \Psiinv$, we see that $\vboldbar_y = i\vboldbar_x$ and
\begin{align*}
\vboldbar_x = \pz\vboldbar = \pz(\F\compose\Psiinv) = \brac{\onePsizp\Fzp}\compose\Psiinv
\end{align*}
Hence by \thmref{thm:mainangle}, we see that $\vbold_x(\cdot,t)$ and $\vbold_y(\cdot,t)$ extend continuously to $\Omegabar(t)$ with $\vbold_x(\z,t) = \vbold_y(\z,t) = 0$ for all $\z \in \cbrac{\Z(\ap,t)\suchthat \ap \in S(t)} \subset \Sigma(t)$. 

\item We know that 
\begin{align*}
(\px -i\py) P = \brac{\frac{1}{\Psizp} (\pxp - i\pyp)\Pfrak}\compose\Psiinv
\end{align*}
Hence by \thmref{thm:gradP}, we see that $\grad P(\cdot,t)$  extends continuously to $\Omegabar(t)$ with $\grad P(\z,t) = 0$ for all $\z \in \cbrac{\Z(\ap,t)\suchthat \ap \in S(t)} \subset \Sigma(t)$. 

\item Observe that 
\begin{align*}
\vboldbar_t = (\F\compose\Psiinv)_t = \brac{\F_t - \Psi_t\frac{\Fzp}{\Psizp}}\compose\Psiinv
\end{align*}
As $\Fbar(\cdot,t)$ and $\frac{\Fzp}{\Psizp}(\cdot,t)$ extends continuously to $\Pminusbar$, from equation \eqref{eq:EulerRiem} we easily see that $\brac{\F_t - \Psi_t\frac{\Fzp}{\Psizp}}(\cdot,t)$ extends continuously to $\Pminusbar$ and hence $\vbold_t(\cdot,t)$ extends continuously to $\Omegabar(t)$.

\item As all the quantities in equation \eqref{eq:Euler} extend continuously to $\Omegabar(t)$, the Euler equation also holds on $\Omegabar(t)$. As $\grad P(\z,t) = 0$ at the singularities, we now see that $ \brac{\mathbf{v_t + (v.\nabla)v}}(\z,t) = -i $ for all  $\z \in \cbrac{\Z(\ap,t)\suchthat \ap \in S(t)} \subset \Sigma(t)$. 

\end{enumerate}

\end{proof}

\section{Examples: Angled crests and cusps}\label{sec:exampleE}

We now give a few examples of interfaces for which our theorems apply. In this section we assume that the interface approaches the real axis rapidly at infinity and construct domains for which
\begin{enumerate}
\item $\dis \sup_{\yp<0}\norm[\Linfty(\Rsp,\diff \xp)]{\frac{1}{\Psizp}(\xp+i\yp,0)} < \infty$
\item $\dis  \sup_{\yp<0}\norm[\Ltwo(\Rsp, \diff \xp)]{\partial_\zp \brac{\frac{1}{\Psizp}}(\xp+i\yp,0)} < \infty \quad $ and 
\item $\dis \sup_{\yp<0}\norm[\Ltwo(\Rsp, \diff \xp)]{\frac{1}{\Psizp}\pzp\brac{\frac{1}{\Psizp}\pzp \brac{\frac{1}{\Psizp}}}(\xp+i\yp,0)} < \infty$
\end{enumerate}
Observe that from these conditions we automatically get $ \sup_{\yp <0}\norm[\Ltwo(\Rsp,\diff \xp)]{\frac{1}{\Psizp(\xp+i\yp,0)} - 1} < \infty$. To control terms involving $\F$ we will assume that 
\begin{align*}
\sup_{\yp <0}\norm[H^3(\Rsp,\diff \xp)]{\F(\xp + i\yp,0)} < \infty
\end{align*}
As an example, $F$ being identically zero satisfies this assumption. It is easy to see that with these assumptions we have $\Ecalone(0)<\infty$ and $c_0<\infty$ and hence \thmref{thm:Wuexistence} applies. 

In this section we will use the following notation: Let $A \subset \Csp$ be a non-empty set and let $p \in \bar{A}$. Let $f,g:A \to \Csp$ be functions such that $g(z) \neq 0$ for all $z$ in a punctured neighborhood of $p$. We say that $f(z)\sim g(z)$  as  $z \to p$ in $A$, if $\lim_{\z \to p}\frac{f(z)}{g(z)} \in \Csp^*$ where $\Csp^* = \Csp\backslash\cbrac{0}$. In this section we will mostly have $A = \Pminus$ and $p = 0$. We will also let $\log(z)$ denote the principle branch of the logarithm.

\bigskip
\noindent \textbf{1. Smooth domains}
\medskip

Observe that if the boundary is of class $C^{1,\al}$ with $0<\al\leq 1$, then there exists constants $0<c_1,c_2 < \infty$ such that $c_1 \leq \Zapabs(\ap,0) \leq c_2$. Hence if in addition a domain satisfies $(\Zap - 1)(\cdot,0) \in H^2(\Rsp)$, then we easily see that $\Ecalone(0) < \infty$. In particular if the initial domain is smooth then  \thmref{thm:Wuexistence} applies. In this case the set of singularities $S(t)$ is empty. Note that in this case our results from \secref{sec:results} still apply but are not novel in any way.

\bigskip
\noindent \textbf{2. Angled crests}
\medskip

A regular smooth curve in the plane is a smooth mapping $\gamma:I\to \Csp$ such that $\gamma'(s) \neq 0$ for all $s\in I$ where $I$ is an interval. Consider a simply connected domain $\Omega \subset \Csp$ with $0\in \partial \Omega$ such that the boundary of $\Omega$ at $0$ consists of two regular smooth arcs such that the opening angle of $\Omega$ at $0$ is $\nu\pi$ i.e. there is a corner at $0\in \partial \Omega$. Assume that $0<\nu<1$. Then we have a local description of $\Psi$ near $z=0$.
\begin{thm}[\cite{Wi65}]\label{thm:Riemsing}
Let  $\Omega$ a domain as described above with $0\in \partial\Omega$. Let $\Psi:\Pminus \to \Omega$ be a Riemann map with $\Psi(0) = 0$ and let $n\geq 0 $ be an integer. Then 
\begin{align*}
\partial_z^n \Psi (z) \sim z^{\nu - n} \qq \tx{ as } \z \to 0 \tx{ in } \Pminus
\end{align*}
\end{thm}
Hence we now consider an interface with $N\geq 1$ angled crests with angles $\nu_i\pi$ where $0<\nu_i<\half$. Near an angled crest, if we change coordinates so that it is at the origin, then by the above result we see that as $z \to 0$ in $\Pminus$ we have
\begin{enumerate}
\item $\dis \frac{1}{\Psizp}(z) \sim z^{1-\nu}$ 

\item $\dis \partial_\zp \brac{\frac{1}{\Psizp}}(z) \sim z^{-\nu}$ 

\item $\dis \frac{1}{\Psizp}\pzp\brac{\frac{1}{\Psizp}\pzp \brac{\frac{1}{\Psizp}}}(z) \sim z^{1-3\nu}$
\end{enumerate}
Hence the energy $\Ecalone(0) < \infty$ and \thmref{thm:Wuexistence} applies. This argument that the energy allows angled crests was also given in \cite{KiWu18}. Note that all the results in \secref{sec:results} apply to such interfaces.

\bigskip
\noindent \textbf{3. Cusps}
\medskip

A regular analytic curve in the plane is an analytic mapping $\gamma:I\to \Csp$ such that $\gamma'(s) \neq 0$ for all $s\in I$ where $I$ is an interval. Consider a domain $\Omega$ with $0\in \partial \Omega$ such that the boundary of $\Omega$ at $0$ consists of two regular analytic arcs such that the opening angle of $\Omega$ at $0$ vanishes i.e. there is a cusp at $0\in \partial \Omega$. By an analytic change of coordinates near $0$ we can assume that the boundary of $\Omega$ near $0$ consists of two arcs one of which is the positive real axis. Assume that there exists an $R>0$ such that
\begin{align}\label{form:Omega}
\Omega\cap \Bbar(0,R) = \cbrac{z\in\Csp \suchthat \abs{z} \leq R, 0<arg(z) < \thvar(\abs{z})}
\end{align}
where $\thvar(s)$ is a real power series that converges on $(-2R,2R) $ and is positive on $(0,R)$.  Then we have the following description of the Riemann mapping near $0$.

\begin{thm}[\cite{KaLe16}]\label{thm:cusp}
Let  $\Omega$ a domain as described above with $0\in \partial\Omega$. Let $\thvar(s) = \sum_{j=1}^{\infty} a_j s^j$
be the power series of $\thvar(s)$ around $0$ with $a_1 \neq 0$ . Let $\Psi:\Pminus \to \Omega$ be a Riemann map with $\Psi(0) = 0$ and let $n\geq 1$ be an integer. Then 
\begin{align*}
\Psi (z) &\sim \frac{1}{\log(\abs{z})} \qq \tx{ as } \z \to 0 \tx{ in } \Pminus  \\
\partial_z^n \Psi (z) &\sim \frac{z^{- n}}{\log(z)^2}  \qquad \tx{ as } \z \to 0 \tx{ in } \Pminus 
\end{align*}
\end{thm}
Hence we now consider an interface with an analytic cusp as described above. Using the result above we see that as $z \to 0$ in $\Pminus$ we have
\begin{enumerate}
\item $\dis \frac{1}{\Psizp}(z) \sim z\log(z)^2$ 

\item $\dis \partial_\zp \brac{\frac{1}{\Psizp}}(z) \sim \log(z)^2$ 

\item $\dis \frac{1}{\Psizp}\pzp\brac{\frac{1}{\Psizp}\pzp \brac{\frac{1}{\Psizp}}}(z) \sim z\log(z)^6$ 
\end{enumerate}
Hence the energy $\Ecalone(0) < \infty$ and hence cusps are allowed in \thmref{thm:Wuexistence}. Note that all the results in \secref{sec:results} apply even in this case. In particular interfaces with cusps at $t=0$, have cusps for $t>0$.

It is important to note that an interface with a cusp is not chord arc and hence the second part of   \thmref{thm:Wuexistence} does not apply. So it seems that it may be possible that an interface with a cusp at $t=0$ may self intersect instantaneously. We now show that this does not happen. For simplicity let us assume that there is a single cusp at $\ap =0$ at $t=0$ of the form as described above. Hence from \eqref{form:Omega} and \thmref{thm:cusp} we see that there exists a $\ep>0$ such that 
\begin{align*}
 \abs{\apone - \aptwo}^\half \leq \abs{\Psi(\apone) - \Psi(\aptwo)} \qq \tx{ for all } \apone,\aptwo \in (-\ep,\ep)
\end{align*}
The following lemma shows that for such interface interfaces, the interface does not self intersect for a short period of time. For simplicity we only state the lemma for a single cusp but it can be easily be generalized to the case of multiple cusps.

\begin{lem}
Let $(\F,\Psi,\Pfrak)$ be a solution in $[0,T_0]$ as given by \thmref{thm:Wuexistence}. Also assume that the initial interface $\Z(\cdot,0)$ is an absolutely continuous non-self intersecting curve and that there exists an $\ep>0$ such that
\begin{enumerate}

\item There exists a $0<\delta<1$ such that for all $\ap < \bp$ with $(\ap,\bp) \in \Rsp^2\backslash (-\ep,\ep)\times(-\ep,\ep) $, we have
\begin{align*}
\delta\int_\ap^\bp \abs{\Zap(\gamma,0)}\diff\gamma \leq \abs{\Z(\ap,0) - \Z(\bp,0)} \leq \int_\ap^\bp \abs{\Zap(\gamma,0)}\diff\gamma 
\end{align*}

\item  $\abs{\ap - \bp}^\half \leq \abs{\Z(\ap,0) - \Z(\bp,0)} \qq \tx{ for all } \ap,\bp \in (-\ep,\ep)$

\end{enumerate}
Then there is $T_1>0$ with $T_1 = C(\Ecalone(0))$, such that on $[0,\min\cbrac*[\big]{T_0,\frac{\delta}{T_1}}]$,  the interface $Z = Z(\cdot,t)$ is an absolutely continuous non-self intersecting curve.
\end{lem}
\begin{proof}
We only need to slightly modify the proof of \thmref{thm:Wuexistence} to prove this lemma. By following the proof of \thmref{thm:Wuexistence}, we see that for all $t\in [0,T_0] $ and for all $\al,\be\in \Rsp$ we have
\begin{align}\label{eq:chordarcestimate}
\abs{\abs{\z(\al,t) - \z(\be,t)} - \abs{\Z(\al,0) - \Z(\be,0)}} \leq t C(\Ecalone(0))\int_\al^\be \abs{\Zap(\gamma,0)}\diff\gamma
\end{align}
which in particular implies
\begin{align*}
\abs{\z(\al,t) - \z(\be,t)}  \leq (t C(\Ecalone(0)) + 1)\int_\al^\be \abs{\Zap(\gamma,0)}\diff\gamma
\end{align*}
which immediately implies that $\z(\cdot,t)$ is absolutely continuous for all $t \in [0,T_0]$ and hence so is $Z(\cdot,t)$. From  \eqref{eq:chordarcestimate} we see that for all $0\leq t \leq \min\cbrac{T_0, \frac{\delta}{2C(\Ecalone(t))}}$ and for all  $\al<\be$  with $(\al,\be) \in \Rsp^2\backslash (-\ep,\ep)\times(-\ep,\ep) $ we have
\begin{align}\label{}
\half\delta\int_\al^\be \abs{\Zap(\gamma,0)}\diff \gamma \leq \abs{\z(\al,t) - \z(\be,t)} \leq 2\int_\al^\be \abs{\Zap(\gamma,0)}\diff \gamma 
\end{align}
Hence we now only need to show that $\z(\al,t)$ and $\z(\be,t)$ do not intersect if $\al,\be \in (-\ep,\ep)$.  We see that
\begin{align*}
\zep(\al,t) - \zep(\be,t) = \zep(\al,0) - \zep(\be,0) + \int_0^t \int_\al^\be \ztalep(\gamma,s)\diff\gamma\diff s
\end{align*}
Observe that for all $t \in [0,T_0] $ we have $\norm[2]{\ztalep(\cdot,t)} \leq C(\Ecalone(0))$. Hence we have
\begin{align*}
\abs{\abs{\zep(\al,t) - \zep(\be,t)} - \abs{\Zep(\al,0) - \Zep(\be,0)}} \leq t C(\Ecalone(0))\abs{\al-\be}^\half
\end{align*}
Now taking the limit as $\ep \to 0$ we obtain
\begin{align}\label{eq:cuspestimate}
\abs{\abs{\z(\al,t) - \z(\be,t)} - \abs{\Z(\al,0) - \Z(\be,0)}} \leq t C(\Ecalone(0))\abs{\al-\be}^\half
\end{align}
Hence we see that for all $0\leq t \leq \min\cbrac{T_0, \frac{\delta}{2C(\Ecalone(t))}}$ we have
\begin{align*}
\half\abs{\al - \be}^\half \leq \abs{\z(\al,t) - \z(\be,t)} \qq \tx{ for all } \al,\be \in (-\ep,\ep)
\end{align*}
From this we see that $\z(\cdot,t)$ is non-self intersecting for all $0\leq t \leq \min\cbrac{T_0, \frac{\delta}{2C(\Ecalone(t))}}$ and hence so is $\Z(\cdot,t)$.
\end{proof}

\medskip

\bibliographystyle{amsplain}
\bibliography{../../../Writing/Main.bib}

\end{document}